\documentclass[11pt,american, reqno]{amsart}

\usepackage{amsmath,amsfonts,amssymb,amscd,amsthm,amsbsy,bbm,epsf,calc,graphicx,esint}
\usepackage{mathtools}
\usepackage{color}
\usepackage[dvipsnames]{xcolor}
\usepackage{bigints}

\usepackage{datetime}
\usepackage{comment}
\usepackage[colorlinks,linkcolor = blue,citecolor = blue]{hyperref}
\allowdisplaybreaks

\usepackage[nobysame]{amsrefs}

\usepackage{mathrsfs}

\usepackage[normalem]{ulem}

\numberwithin{equation}{section}

\newtheorem{theorem}{Theorem}[section]
\newtheorem*{thm}{Main Theorem}

\newtheorem{lemma}[theorem]{Lemma}
\newtheorem{definition}[theorem]{Definition}
\newtheorem{remark}{Remark}

\newcommand{\N}{\mathbb N}
\newcommand{\R}{\mathbb R}

\newcommand{\la}{\langle}
\newcommand{\ra}{\rangle}

\newcommand{\Ni}{\noindent}

\newcommand{\del}{\partial}

\newcommand{\ZZ}{{\mathbb{Z}}}

\newcommand{\ds}{\, {\rm d} s}
\newcommand{\dt}{\, {\rm d} t}
\newcommand{\dT}{\, {\rm d} T}
\newcommand{\dv}{\, {\rm d} v}
\newcommand{\dw}{\, {\rm d} w}
\newcommand{\dx}{\, {\rm d} x}
\newcommand{\dy}{\, {\rm d} y}
\newcommand{\dz}{\, {\rm d} z}
\newcommand{\CalM}{\mathcal{M}}
\newcommand{\CalA}{\mathcal{A}}
\newcommand{\CalF}{\mathcal{F}}

\newcommand{\CalT}{\mathcal{T}}

\newcommand{\dxi}{\, {\rm d} \xi}
\newcommand{\deta}{\, {\rm d} \eta}
\newcommand{\dtau}{\, {\rm d} \tau}
\newcommand{\dtheta}{\, {\rm d} \theta}
\newcommand{\dzeta}{\, {\rm d} \zeta}

\newcommand{\One}{{\boldsymbol 1}}

\newcommand{\vpran}[1]{\left(#1\right)}
\newcommand{\nn}{\nonumber}
\newcommand{\Eps}{\epsilon}

\newcommand{\norm}[1]{\left\lVert#1\right\rVert}
\newcommand{\abs}[1]{\left\vert#1\right\vert}

\newcommand{\vint}[1]{\left\langle#1\right\rangle}
\newcommand{\T}{\mathbb{T}}

\numberwithin{equation}{section}

\title[Uniqueness]{The $\CalM$-Operator and Uniqueness of Nonlinear Kinetic Equations}

\author{Ricardo Alonso}
\address{Division of Arts and Sciences, Texas A\&M University at Qatar, Education city, Doha}
\email{ricardo.alonso@qatar.tamu.edu} 

\author{Maria Pia Gualdani}
\address{Department of Mathematics, The University of Texas at Austin, 2515 Speedway, Austin TX, 78712}
\email{gualdani@math.utexas.edu}
\thanks{}
\author{Weiran Sun}
\address{Department of Mathematics, Simon Fraser University, 8888 University Dr., Burnaby, BC, Canada V5A 1S6}
\email{weirans@sfu.ca}

\begin{document}
\begin{abstract}
We introduce an $\CalM$-operator approach to establish the uniqueness of continuous or bounded solutions for a broad class of Landau-type nonlinear kinetic equations. The specific $\CalM$-operator, originally developed in~\cite{AMSY2020}, acts as a negative fractional derivative in both spatial and velocity variables and interacts in a controllable manner with the kinetic transport operator. The novelty of this method is that it bypasses the need for bounds on the derivatives of the solution - a condition typically sought after in uniqueness arguments for non-cutoff equations. As a result, our method enables us to work with solutions with low regularity. 
\end{abstract}
\maketitle

\section{Introduction} 
The uniqueness of solutions with low regularity or, equivalently, with rough initial data is a central and challenging question in the analysis of nonlinear kinetic equations.  There are several unsolved problems regarding uniqueness; one of the most relevant is the uniqueness of renormalized solutions for the inhomogeneous Boltzmann equation \cite{DL}.  
In the context of the inhomogeneous cutoff Boltzmann equation close to equilibrium, one of the most general uniqueness results can be found in \cite{GMM}, which relates to the uniqueness of bounded solutions. In contrast, the analogous question of uniqueness of bounded solutions to the non-cutoff Boltzmann or Landau equations remains open. The most general results require H\"older regularity in the spatial variable~\cite{AMUXY2011,DLSS2020, HST2020, HST2025, HW2024}. Reducing the regularity of solutions to mere boundedness is the key motivation of the present work.

In this paper, we will focus on inhomogeneous Landau-type equations and establish the uniqueness within the classes of continuous or bounded solutions. More specifically, we study two types of equations. The first is a toy model 
\begin{align} \label{eq:toy-original-1}
 \del_t f + v \cdot \nabla_x f 
&= \nabla_v \cdot \vpran{\rho(t, x) \vint{v}^\beta \nabla_v f}, 
\qquad
  \rho(t, x) =  \int_{\R^3} f(t, x, v) \dv, 
\end{align}
and the second one is a viscous  Landau-Coulomb equation 
\begin{align}\label{FPL-1}
\partial_t f + v \cdot \nabla_x f &= \nabla_v \cdot \vpran{A[f] \nabla_v f - f \nabla_v a[f]} +\nu \Delta_v f,
\end{align}
with $\nu > 0$ and $A[f], \; a[f]$ defined as
\begin{align*}
A[f] (x,v,t)& := \frac{1}{8\pi} \int_{\mathbb{R}^3} \frac{\mathbb{P}(v-z) }{|v-z|} f(x,z,t)\;dz,
\qquad 
\mathbb{P}(z) := \mathbb{I}d - \frac{z \otimes z}{|z|^2},
\\
a[f] (x,v,t) &:= \frac{1}{4\pi} \int_{\mathbb{R}^3} \frac{1 }{|v-z|} f(x,z,t)\;dz.
\end{align*}

The main results in this paper can be informally summarized as follows:
\begin{thm}

 \Ni (a) Let $\beta \leq 2$. Equation ~\eqref{eq:toy-original-1} has a unique solution that is continuous in $(t,x,v)$ and has a fast enough polynomial decay tail in $v$.  

\medskip

\Ni (b) Let $\nu > 0$ be arbitrary. Equation ~\eqref{FPL-1} has a unique solution that is continuous in $(t,x,v)$ with a fast enough polynomial decay tail in $v$.

\medskip
\Ni (c) Let $\nu > 0$ be arbitrary and $k_0$ large enough. Suppose $\norm{\vint{v}^{k_0} f_0}_{L^\infty_{x,v}}$ is small enough with its bound depending on $\nu$. Then~\eqref{FPL-1} has a unique solution that is bounded in $(t,x,v)$ with a sufficiently fast polynomial decay tail in $v$. 
\end{thm}

The precise statements can be found in Theorem~\ref{thm:unique-w-toy}, Theorem~\ref{Thm_Landau_0} and Theorem~\ref{Thm_Landau}. 

\smallskip

In the literature, a common approach to proving the uniqueness of solutions to nonlinear diffusion equations is to identify a suitable function space in which some derivatives in $x$ or $v$ of the solution are bounded. Using~\eqref{eq:toy-original-1} as an example, the difference between two solutions with same initial data, $w := f - g$, satisfies 
\begin{align*}
 \del_t w + v \cdot \nabla_x w 
&= \nabla_v \cdot \vpran{\rho_f(t, x) \vint{v}^\beta \nabla_v w}
+ \nabla_v \cdot \vpran{\rho_w(t, x) \vint{v}^\beta \nabla_v g}, 
\\
 w|_{t=0} & = 0,
\end{align*}
where $\rho_f(t, x) =  \int_{\R^3} f(t, x, v) \dv$ and $\rho_w(t, x) =  \int_{\R^3} w(t, x, v) \dv$. In a basic energy estimate, the last term would need a uniform bound in the variable $x$ for $\nabla_v g$. Several uniqueness results for inhomogeneous kinetic equations present in the literature follow this approach.  The most recent ones on the inhomogeneous Landau or non-cutoff Boltzmann equations are included in \cites{AMUXY2011,DLSS2020, HST2020, HST2025, HW2024}. In~\cite{AMUXY2011}, it is shown that the non-cutoff Boltzmann equation with $2s$-singularity has unique solution in $L^\infty_{t,x}H^{2s}_v$. This corresponds to $L^\infty_{t,x}H^{2}_v$ for the Landau equation. In~\cite{DLSS2020}, uniqueness of solutions to the non-cutoff Boltzmann and Landau equations is shown in the class of functions near the global Maxwellian in the space $X = L^1_k L^\infty_T L^2_v$ composed of the Wiener algebra. The space $X$ is stronger than $C^{1/2}_x L^\infty_T L^2_v$, where $C^{1/2}_x$ denotes the H\"older space of order $1/2$ in $x$. In recent work \cite{HST2020, HST2025}, kinetic Schauder estimates have been employed to prove the uniqueness of solutions to the non-cutoff Boltzmann and Landau equations if the initial data have an arbitrary amount of H\"older regularity in both $x$ and $v$. These results have recently been improved by~\cite{HW2024} to initial data with H\"older regularity in $x$ and logarithmic modulus of continuity in $v$.

In this paper, we take a different point of view and make use of negative Sobolev spaces. To illustrate the idea, we consider again the toy model~\eqref{eq:toy-original-1} with $\beta = 0$ without transport term:
\begin{align} \label{eq:toy-simp}
 \del_t f 
= \Delta_v \vpran{\rho(t, x)  f}.
\end{align}
We now use $(-\Delta_v)^{-1} f$ as a test function. This is effectively equivalent to making $H^{-1}$ estimates on 
\begin{align*} 
 \del_t (-\Delta_v)^{-1} f 
= \rho(t, x)  f,
\end{align*}
for which the derivative on $f$ does not need to be bounded. Unfortunately, direct application of $(-\Delta_v)^{-1}$ to the kinetic transport term no longer works, as its commuter with $v \cdot \nabla_x$ induces differentiation in $x$. 

To address this issue, we propose to use a pseudo-differential operator $\mathcal{M}$ that plays the role of $(-\Delta_v)^{-1}$ and commutes with $v \cdot \nabla_x$ in a controllable way. This operator will feature a time-averaged {\em{negative}} Sobolev derivative with intermediate time. Operators with a time-averaged {\em{positive}} Sobolev derivatives already appeared in studies of the well-posedness of nonlinear kinetic equations (see, for example, \cite{MX2020}). The pseudo-differential operator we use here is inspired by the one introduced in~\cite{AMSY2020}, used to prove the $L^2_x H^{-s}_v$ to $L^2_x L^2_v$ regularization estimates for the linearized Boltzmann equation. Let $\widehat{w}(\eta,\xi) $ be the Fourier transform of $w$ in $(x, v)$. We consider the Fourier multiplier operator with the symbol ${M}$ as 
\begin{align} \label{def:M}
{M}(t,T, \eta,\xi) := \frac{1}{\left( 1 + \delta \int_t ^T \langle \xi +(t- \tau) \eta\rangle^{2} \;d\tau\right)^{\frac{1}{2}+ \Eps}}.
\end{align}
and define ${\mathcal{M}}$ as 
$$
{\mathcal{M}} w:= ({\mathcal{F}_{(x,v) \to (\xi, \eta)}})^{-1}({M} \widehat{w} ).
$$
Here, $\Eps$ and $\delta$ are positive constants that can be chosen as needed. Precise properties for (\ref{def:M}) are summarized in Section~\ref{Sec:M-operator}. However, let us analyze the main characteristics. The commutator of $\CalM$ with $v \cdot \nabla_x$ generates a backward diffusion in $v$, whose strength is proportional to the parameter $\delta$: 
\begin{align} \label{intro:commut}
[\CalM, \del_t + v \cdot \nabla_x]  
\sim  
- \delta (I - \Delta_v) \mathcal{M}.
\end{align}
This commutator term will be controlled by the diffusion on the right-hand side of~\eqref{eq:toy-original-1} or by viscosity term of~\eqref{FPL-1}. 
In the case of (\ref{eq:toy-original-1}) the loss of weights (for $\beta <0$) can be overcome by building a negative velocity weight in $\CalM$: in Section~\ref{Sec:toy} we introduce a dyadic decomposition in $M$ to localize the velocity in each ring where $|v| \in [2^n, 2^{n+2}]$ for each $n \in \N$. 
Unfortunately, the dyadic decomposition approach is not yet applicable to the anisotropic diffusion in the original Landau-Coulomb equation; therefore, in the proof of the uniqueness of the solution to (\ref{FPL-1}) we use the added viscosity $ \nu \Delta_v$ to control (\ref{intro:commut}).

This work makes two main contributions. First, it addresses the longstanding question regarding the largest function spaces in which uniqueness holds for the original non-cutoff Boltzmann or Landau equations. Given the two examples of nonlinear inhomogeneous equations presented in this work, it is reasonable to believe that the uniqueness of bounded solutions may indeed hold for the Boltzmann and Landau equations.

Second, we introduce a novel approach for studying uniqueness properties of inhomogeneous kinetic equations by employing pseudo-differential operators acting on negative Sobolev spaces in the $(x, v)$ variables. The $\CalM$-operator carefully takes into account the specific structure of the transport operator and is designed to work well with negative Sobolev spaces. Although the two models considered here have structural advantages compared to the original Landau or Boltzmann equations, existing strategies from previous works, based on direct energy methods or Schauder estimates, do not seem sufficient to establish uniqueness within the same regularity classes considered here in the Main Theorem. 

To further highlight the potential of using negative Sobolev spaces, we recall a study on the homogeneous Landau-Coulomb equation~\cite{HJJ2025}, which was brought to our attention very recently. In~\cite{HJJ2025}, the uniqueness (together with the global existence) is shown for initial data in a weighted space $H^{-r} \cap L^1 \cap L\log L$ with $r \geq 1/2$. The main techniques used in~\cite{HJJ2025} include a careful analysis using Littlewood-Paley type estimates based on localization in the frequency space. A very interesting result in~\cite{HJJ2025} is the novel uniqueness proof for the boardeline case $r=1/2$.  In particular, this new estimate avoids the standard route of applying the framework in~\cite{Fournier10}, which is based on establishing that the solution lies in $L^1(0, T; L^\infty(\R^3))$. It is unclear, however, whether this approach can be easily adapted to inhomogeneous equations; the localized pseudo-differential operators used in \cite{HJJ2025} do not commute well with the transport operator.


The resemblance of systems~\eqref{eq:toy-original-1} and~\eqref{FPL-1} to the original non-cutoff Boltzmann and the Landau-Coulomb equations, along with the success in~\cite{HJJ2025} strengthens our belief that the use of negative Sobolev spaces is a viable approach to tackling the uniqueness problems of the original kinetic equations. Although we only consider Landau-type equations in the present work, preliminary analysis suggests that our analysis also  applies to non-cutoff  Boltzmann equations with a viscous term. The particular structure of the collision operator does not appear to be essential. A full analysis is left for future work.  

Let us conclude by mentioning additional uniqueness results available for homogeneous kinetic equations. As mentioned before, for the Landau-Coulomb equation,  \cite{Fournier10} shows that if the solution belongs to $L^1(0,T,L^{\infty}(\mathbb{R}^d))$, then it is unique.  In addition, it is known that there exists a unique smooth solution if the initial data belongs to $L^p(\mathbb{R}^d)$ for $p>d/2$, see \cite{GGL25}. 
%
For hard potentials, uniqueness was established in \cite{DV00} and, more recently, in \cite{FH21} for more general initial data. For moderately soft potential \cite{ALL15} shows the uniqueness of smooth solutions.  For the homogeneous Boltzmann equation refer to \cite{Fournier08} and the references therein for uniqueness results in both the cutoff and non-cutoff setting. Analysis of inhomogeneous kinetic equations similar to ~\eqref{eq:toy-original-1} can be found in \cite{Ancheschi-Zhu}, \cite{LWY18} and \cite{Imbert-Mouhot}.

\subsection{Notation} The bracket $\langle v \rangle $ is a short notation for $ (1+ |v|^2)^{1/2}$. We denote
\begin{align*}
 L^\infty(\T^3; L^q(\R^3))
&=L^\infty_x L^q_v
= \left\{f | \, \text{esssup}_{x \in \T^3} \vpran{\int_{\R^3} \abs{f(x, v)}^q \dv}^{1/q} < \infty \right\},
\\
 L^\infty(\T^3; L^\infty_{k_0}(\R^3))
&=L^\infty_x L^\infty_{k_0}
= \left\{f | \, \text{esssup}_{(x, v) \in \T^3 \times \R^3} \abs{\vint{v}^{k_0} f(x, v)} < \infty \right\},
\\
 L^\infty([0, T] \times \T^3; L^q_{k_0}(\R^3))
&=L^\infty_{t, x} L^q_v
= \left\{f | \, \text{esssup}_{t \in [0, T], x \in \T^3} \norm{\vint{v}^{k_0} f(t,x,\cdot)}_{L^q(\R^3)} < \infty \right\}.
\end{align*}
Similar notations apply to $L^2_{t,x,v}$ and $L^2_{T,t,x,v}$, where the subindices encode the information of the integration domains. For presentation purposes, we often suppress the volume elements $\dx\dv\dt$ or $\dx\dv\dt\dT$ in long integrals. In such cases, the integration domains will be explicitly stated to avoid any ambiguity. We also use $\nabla$ to denote $\nabla_v$, the gradient in the $v$-variable. The $x$-gradient is specified as $\nabla_x$.

\section{Properties of the $\CalM$-operator and technical lemmas} \label{Sec:M-operator}
In this section, we state the main properties of the $\CalM$-operator, and  prove several technical lemmas. Let $\mathcal{F}$ denote the Fourier transform in $(x, v)$. We define the Fourier multiplier $\CalM$ with symbol ${M}$ as 
$$
{\mathcal{F}}({\mathcal{M}} w):= {M} {\mathcal{F}}(w) ,
$$
with   
\begin{align*} 
{M}(t,T, \eta,\xi) := \frac{1}{\left( 1 + \delta \int_t ^T \langle \xi +(t- \tau) \eta\rangle^{2} \;d\tau\right)^{\frac{1}{2}+ \Eps}},
\qquad 
\Eps, \delta > 0.
\end{align*}
In the next lemma, we summarize the key properties of $\CalM$ and its symbol $M$. The Fourier transform of a function $h$ in $x, v$ is denoted as $\hat h$.
\begin{lemma} 
\label{lem:M-basic}
Let $h$ be any smooth enough function; we have  
\begin{align*}
\| \CalM h\|_{L^2_{x,v}} \le \|h\|_{L^2_{x,v}},
\end{align*}
\begin{align}
{\CalM} (\partial_t  h   -v \cdot \nabla_x   h)  
& = \vpran{\partial_t - v \cdot \nabla_x } ({\CalM} h) + [\CalM, \partial_t - v \cdot \nabla_x ]h,\label{trasp_comm}
\end{align}
where 
$$
\mathcal{F}([\CalM, \partial_t  h   -v \cdot \nabla_x ]h) = - \delta \left(\tfrac{1}{2}+ \Eps\right)\tfrac{\langle \xi \rangle^{2}}{1 + \delta \int_t ^T \langle \xi +(t- \tau) \eta\rangle^{2} \;d\tau} {M}\mathcal{F}({h}), 
$$  
\begin{align}
 & \int_t^{T_0}  \langle \xi \rangle ^2 {M}^2(t,T, \eta,\xi)  \dT 
\le    
 \frac{2}{\Eps \delta}, \label{int_xi_M2}
\\
 \int_0^{T_0}  \int_0^{T}  \int_{\R^3}  &|\xi|^2 | {M} \hat{h}| ^2d\xi  \dt \dT \le \frac{2}{\Eps \delta} \int_0^{T_0} \int_{\R^3} | \hat{h} |^2  \dxi \dt.  \label{lapl}
\end{align}
\end{lemma} 

\begin{proof} 
The first inequality is the consequence of the Plancherel identity and the fact that $M\le 1$: 
$$
\| \CalM h \|_{L^2_{x,v}} = \| M \hat{h} \|_{L^2_{\xi, \eta}} \le  \| \hat{h} \|_{L^2_{x,v}}. 
$$
Product rule and direct calculations yield (\ref {trasp_comm} ) since 
\begin{align*}
{M} (\partial_t \widehat h   -\eta \cdot \nabla_\xi \widehat  h)  =  \partial_t ({M}\widehat h)  -\eta \cdot \nabla_\xi ({M}\widehat  h)  -\widehat  h( \partial_t {M} -\eta \cdot \nabla_\xi {M})
\end{align*}
and 
\begin{align*}
 \partial_t {M} -\eta \cdot \nabla_\xi {M} = \delta \left(\tfrac{1}{2}+ \Eps\right)\frac{ \langle \xi \rangle^{2}}{1 + \delta \int_t ^T \langle \xi +(t- \tau) \eta\rangle^{2} \dtau} {M}. 
\end{align*}
For the second estimate, we first notice that  
\begin{align*}
 1 + \delta \int_t ^T \langle \xi +(t- \tau) \eta\rangle^{2} \dtau   
\geq  
  1+\frac{1}{4} \delta \langle \xi \rangle ^2 (T-t),
\end{align*}
which implies 
\begin{align*}
 \langle \xi \rangle ^2 {M} 
\le 
  \frac{ \langle \xi \rangle ^2}{ \left( 1+\frac{1}{4} \delta  \langle \xi \rangle ^2 (T-t)\right)^{\frac{1}{2}+\Eps}}, 
\end{align*}
and consequently 
\begin{align*}
 \int_t^{T_0}  \langle \xi \rangle ^2 {M}^2 \dT 
\le   
 \frac{2}{\Eps \delta}. 
\end{align*}
The last inequality follows from the previous one after switching the order of integration:
\begin{align*}
\int_0^{T_0}  \int_0^{T}  \int  |\xi|^2 | {M}\widehat h| ^2 \dxi  \dt \dT & = \int_0^{T_0} \int_{\R^3} |\widehat{h}| ^2  \int_t^{T_0}  |\xi|^2 | {M}| ^2 \dT \dxi  \dt \\
& \le \frac{2}{\Eps \delta} \int_0^{T_0} \int_{\R^3} |\widehat h |^2 \dxi \dt. \qedhere
\end{align*}
\end{proof} 

As a preparation for the commutator estimates, we recall the bounds of the derivatives of $\CalM$ in~\cite{AMSY2020}.

\begin{lemma}\label{lem:deriv-M}
Let $M(t,\eta, \xi)$ be the symbol defined in~\eqref{def:M} with $T < 1$. Then for any $\alpha, \beta \in \ZZ^3$ and $\beta \neq 0$, there exists a constant $C$ that only depends on $\Eps, \alpha, \beta$ such that
\begin{align}
\abs{\frac{\nabla_\xi^\alpha  M(t, \eta, \xi)}{M(t, \eta, \xi)}}
&\leq C
 (\la \xi \ra  + (T-t) |\eta| )^{-\min\{|\alpha|, 2\}}, \label{diff:M-xi}
\\
\abs{\frac{\nabla_\eta^\beta  M(t, \eta, \xi)}{M(t, \eta, \xi)} }
\leq C T
& (\la \xi \ra  + (T-t) |\eta| )^{-\min\{|\beta|, 2\}},
\qquad \beta \neq 0.
\label{diff:M-eta}
\end{align}
\begin{proof}
The first inequality~\eqref{diff:M-xi} is shown in Lemma 5.1 of~\cite{AMSY2020}. The second inequality~\eqref{diff:M-eta} follows from the observation that
\begin{align*}
 (t - \tau) \nabla_\xi \vint{\xi + (t - \tau) \eta}
= \nabla_\eta \vint{\xi + (t - \tau) \eta},
\end{align*}
which gives the extra $T$ in~\eqref{diff:M-eta}.
\end{proof}

\end{lemma}

Next, we show three key commutator estimates of the $\CalM$-operator with smooth function in $x$ or $v$ or $(x, v)$. 
\begin{lemma} \label{lem:M-comm-x}
Suppose $\phi \in H^5(\T^3)$. Then for any $f \in L^2(\T^3 \times \R^3)$, we have
\begin{align*}
  \CalM (\phi f) - \phi \CalM f
= \CalA_\phi \CalM f,
\end{align*}
where the operator $\CalA_\phi$ satisfies
\begin{align} \label{bound:toy-commut}
  \norm{\CalA_\phi}_{L^2_{x, v} \to L^2_{x,v}} 
\leq
  C \, T \norm{\phi}_{H^5},
\end{align}
with $T$ being the time in $\CalM$ and $C$ being a generic constant. 
\end{lemma}
\begin{proof}
We show the proof in the Fourier space. Let $\CalF$ be the Fourier transform in $(x, v)$. Then
\begin{align*}
& \quad \,
  \vpran{\CalF \vpran{\CalM (\phi f) - \phi \CalM f}}(\xi, \eta)
\\
&= \tfrac{1}{\vpran{1 + \delta \int_0^{T-t} \vint{\xi - \tau \eta}^2 \dtau}^{\frac{1}{2}+\Eps}}
\sum_{k \in \ZZ^3} \phi_k f_{\eta-k} (\xi)
- \sum_{k \in \ZZ^3} \phi_k f_{\eta-k} (\xi)
\tfrac{1}{\vpran{1 + \delta \int_0^{T-t} \vint{\xi - \tau (\eta-k)}^2 \dtau}^{\frac{1}{2}+\Eps}}
\\
& = \sum_{k \in \ZZ^3} \phi_k f_{\eta-k} (\xi)
\underbrace{\vpran{\tfrac{1}{\vpran{1 + \delta \int_0^{T-t} \vint{\xi - \tau \eta}^2 \dtau}^{\frac{1}{2}+\Eps}}
- \tfrac{1}{\vpran{1 + \delta \int_0^{T-t} \vint{\xi - \tau (\eta-k)}^2 \dtau}^{\frac{1}{2}+\Eps}}}}_{I(\xi, \eta, k)},
\end{align*}
where $\phi_k$ and $f_{\eta-k} (\xi)$ are the corresponding Fourier mode. We will show that for each $\xi, \eta$,
\begin{align*}
  I(\xi, \eta, k) = a(\xi, \eta, k) M(\xi, \eta - k),
\end{align*}
where $a(\xi, \eta, k)$ satisfies
\begin{align} \label{bound:toy-M-diff}
  |a(\xi, \eta, k)| 
\leq 
  C T \vint{k}^3, 
\end{align}
where $C$ is independent of $\xi, \eta, k, \delta, \epsilon$. To this end, write
\begin{align*}
 I (\xi, \eta, k)
&= \int_0^1 
 \frac{\rm d}{\ds} 
\vpran{\tfrac{1}{\vpran{1 + \delta \int_0^{T-t} \vint{\xi - \tau (\eta - (1-s) k)}^2 \dtau}^{\frac{1}{2}+\Eps}}} \ds
\\
&= \vpran{\tfrac{1}{2} + \Eps}
   \bigintsss_0^1
   \frac{2\delta \int_0^{T-t} 
     \tau \vpran{\xi - \tau (\eta - (1-s) k)} \cdot k \dtau}{\vpran{1 + \delta \int_0^{T-t} \vint{\xi - \tau (\eta - (1-s) k)}^2 \dtau}^{\frac{3}{2}+\Eps}} \ds.
\end{align*}
Therefore,
\begin{align*}
 \abs{M^{-1} (\xi, \eta - k) I (\xi, \eta, k)}
\leq
  C \, T \, |k| 
  \bigintsss_0^1
  \frac{\vpran{1 + \delta \int_0^{T-t} \vint{\xi - \tau (\eta - k)}^2 \dtau}^{\frac{1}{2}+\Eps}}{\vpran{1 + \delta \int_0^{T-t} \vint{\xi - \tau (\eta - (1-s) k)}^2 \dtau}^{\frac{1}{2}+\Eps}} \ds ,
\end{align*}
where $C$ is a generic constant independent of $\xi, \eta, k, T, \delta$. By Corollary A.1 in~\cite{AMSY2020}, the term in the integrand satisfies
\begin{align*}
 \frac{1 + \delta \int_0^{T-t} \vint{\xi - \tau (\eta - k)}^2 \dtau}{1 + \delta \int_0^{T-t} \vint{\xi - \tau (\eta - (1-s) k)}^2 \dtau}
& \leq  
 C\frac{1 + \delta \vpran{|\xi|^2 + (T-t)^2 |\eta - k|^2} (T-t)}{1+ \delta \vpran{|\xi|^2 + (T-t)^2 \abs{\eta - (1-s) k}^2} (T-t)}
\\
& \leq
  C \vpran{1 + \frac{|\eta - k|^2}{1 + |\eta - k + s k|^2}} 
\\
& \leq C (1 + |k|^2) ,
\end{align*}
where $C$ is a generic constant. The last step holds since 
\begin{align*}
 \frac{|\eta - k|^2}{1 + |\eta - (1-s) k|^2}
\leq
\begin{cases}
  4 |k|^2,  & \text{if } |\eta - k| \leq 2 |k|, 
\\[3pt]
  5, & \text{if } |\eta - k| > 2 |k|. 
\end{cases} 
\end{align*}
Therefore, we have
\begin{align*}
  \abs{M^{-1} (\xi, \eta - k) I (\xi, \eta, k)}
\leq
  C \, T  \vint{k}^3,
\end{align*}
or equivalently, \eqref{bound:toy-M-diff} holds. By Plancherel Theorem, 
\begin{align*}
  \norm{\CalM (\phi f) - \phi \CalM f}_{L^2_{x,v}}^2
&= \int_{\R^3} \sum_{\eta \in \ZZ^3}
  \abs{\CalF \vpran{\CalM (\phi f) - \phi \CalM f}}^2 \dxi 
\\
& = \int_{\R^3} \sum_{\eta \in \ZZ^3}
  \abs{\sum_{k \in \ZZ^3} \phi_k f_{\eta - k} I(\xi, \eta, k)}^2 \dxi
\\
& \leq
 C \, T^2 \int_{\R^3} \sum_{\eta \in \ZZ^3}
  \abs{\sum_{k \in \ZZ^3} \abs{\phi_k} \abs{M(\xi, \eta - k) f_{\eta - k}} \vint{k}^3}^2 \dxi
\\
& \leq
  C \, T^2 \int_{\R^3} \sum_{\eta \in \ZZ^3}
  \vpran{\sum_{k \in \ZZ^3} |\phi_k|^2\vint{k}^{10}}
  \vpran{\sum_{k \in \ZZ^3} (\CalM f)_{\eta - k}^2 \vint{k}^{-4}} \dxi
\\
& \leq  
 C \, T^2 \norm{\phi}_{H^5}^2 \norm{\CalM f}_{L^2_{x,v}}^2,
\end{align*}
which gives the desired bound in~\eqref{bound:toy-commut}.
\end{proof}

For the commutator with $v$-dependent functions, we recall the following lemma from~\cite{AMSY2020}.
\begin{lemma} \label{lem:M-comm-v}
For any $k \in \R$ there exists an $L^2$-bounded operator $A(v, D_v)$ whose symbol belongs
to $S^{-1}_{0,0}$ uniformly with respect to $\eta, \delta, \Eps, T-t$
such that 
\begin{align*}
&\vint{v}^k [\CalM, \vint{v}^{-k}] = [\vint{v}^k , \CalM] \vint{v}^{-k} = A(v, D_v) \CalM
\end{align*}
with
\begin{align*}
\|\vint{v} ( 1+ |D_v|^2)^{1/2}  A(v,D_v) g\|_{L^2_v} 
\leq C \|g\|_{L^2_v}, \quad 
\forall g \in \mathcal{S}(\R^3),
\\[3pt]
\|\vint{v} A(v,D_v) (1+ |D_v|^2)^{1/2}   g\|_{L^2_v} \le C \|g\|_{L^2_v}, \quad 
\forall g \in \mathcal{S}(\R^3),
\end{align*}
where $C >0$ is independent of $\delta, \Eps, T-t$.  
\end{lemma}

In general, we have the following commutator estimate. 
\begin{lemma} \label{lem:M-comm-x-v}
Suppose $h \in C^\infty_b(\T^3 \times \R^3)$. Then there exists an $L^2$-bounded operator $\CalT$ whose symbol belong
to $S^{-1}_{0,0}$ uniformly with respect to $\delta, \Eps, T-t$
such that 
\begin{align*}
 [\CalM, h] = \CalT \CalM,
\qquad
\CalT = a(x, v, D_x, D_v) \vint{D_v}^{-1},
\end{align*}
where $T$ is the time in $\CalM$ and $a$ is a bounded operator in $L^2_{x,v}$ with bounds independent of $\delta, \epsilon, T, t$.
\end{lemma}
\begin{proof} [Sketch of Proof]
The proof is similar to that of Lemma 5.6 in \cite{AMSY2020} and we only explain the steps involving (small) differences. By~\cite{RuTu2010, TuVa1998}, the symbols in $\T^3 \times \mathbb{Z}^3$ can be extended to $\T^3 \times \R^3$. Hence the calculus in $x \in \T^3$ is similar as in $v \in \R^3$. 
It follows from
Theorem 3.1 of \cite{Kuma1982} that for any $N \in \N$,
\begin{equation}\label{2.2}
\CalM h = h \CalM + \sum_{1 \leq |\alpha| <
N} \frac{1}{\alpha!} h_{(\alpha)} \CalM^{(\alpha)} +
r_N(x, v, \eta, \xi),
\end{equation}
where $h_{(\alpha)} (x, v)= D_{x,v}^{\alpha} h(x, v)$,
$\CalM^{(\alpha)}(\xi)= \partial_{\eta,\xi}^{\alpha}\CalM(\eta,\xi)$
and
\begin{align*}
r_N(x, v,\eta, \xi) = N \sum_{|\alpha| = N} \int_0^1
\frac{(1-\tau)^{N-1}}{\alpha !} r_{N,\tau,\alpha}(x, v,\eta, \xi)
\dtau.
\end{align*}
Here
\begin{align*}
r_{N,\tau,\alpha}(x,v,\eta,\xi) = \mbox{Os}- \int \int e^{-iz\cdot \zeta} e^{-i y \cdot \theta}
\CalM^{(\alpha)}(\eta + \tau \theta, \xi+ \tau \zeta) h_{(\alpha)}(x + y, v+z)
\frac{\dy\dz\dtheta\dzeta}{(2\pi)^6},
\end{align*} where ``Os-" means the
oscillating integral, see for example \cite{Kuma1982}. Write the multi-index $\alpha$ as $\alpha = \alpha_1 + \alpha_2$ where $\CalM^{(\alpha)} = \nabla_\xi^{(\alpha_1)} \nabla_\eta^{(\alpha_2)} \CalM$. Separate the cases where $\alpha_2 = 0$ and otherwise. The main difference is that we will use Lemma~\ref{lem:deriv-M} to estimate $\CalM^{(\alpha)}$ instead of Lemma 5.2 in~\cite{AMSY2020}. For $\alpha_2 \neq 0$, this gives the operators containing an extra factor $T$, which is combined into the operator $T$. The rest is identical to the proof of Lemma 5.6 in~\cite{AMSY2020} and is omitted.   
\end{proof}

The rest of this section contains two technical lemmas. 
First, we explain the dyadic decomposition and its properties applied to the velocity variable in $\R^3$. This technique is useful when dealing with weights in kinetic equations. The construction is classical, but we include the details here for the convenience of the reader. Denote $B(x, r) \subseteq \R^3$ as the open ball centered at $x$ with radius $r$ and $\overline {B(x, r)}$ as its closure. 

\begin{lemma} \label{lem:dyadic}
Let $V_k \subseteq \R^3$ be the dyadic rings defined by
\begin{align*}
 V_0 = B(0, 4),
\qquad 
 V_k = B(0, 2^{k+2}) \setminus \overline{B(0, 2^{k})}, 
\qquad 
 k \geq 1.
\end{align*}
Then there exists a partition of unity $\theta_k(v)$ such that for any $k \geq 0$,
\begin{align} \label{cond:theta-k-1}
 0 \leq \theta_k \leq 1, 
\quad
\text{Supp}(\theta_k) \subseteq V_k,
\qquad
\sum_{k=0}^\infty \theta_k (v) = 1, 
\qquad
\sum_{k=0}^\infty \theta_k^2 (v) \geq  \frac{1}{3}, 
\quad \forall v \in \R^3.
\end{align}
Moreover, there exist two constants $C_1, C_2$ independent of $k$ such that
\begin{align}
\label{cond:theta-k-2}
 \abs{\nabla_v \theta_k}
\leq  
 C_1 2^{-k},
\qquad
 \abs{\nabla_v^2 \theta_k}
\leq  
 C_2 2^{-2k},
\quad  
 \forall k \geq 0. 
\end{align}
\end{lemma}
\begin{proof}
We will choose $\theta_k$ to be radial. To this end, denote $I_k$ as the interval:
\begin{align*}
 I_0 = [0, 4],
\qquad
I_k = [2^k, 2^{k+2}],
\quad
k \geq 1. 
\end{align*}
Let $\phi_0 \in C^\infty[0, \infty)$ satisfy that 
\begin{align*}
 0 \leq \phi_0 \leq 1,
\qquad
\text{Supp}(\phi_0) \subseteq [0, 4],
\qquad
\phi_0 (r) = 1 \,\, \text{for} \,\, 0 \leq r \leq 3.
\end{align*}
Let $\phi_1 \in C^\infty[0, \infty)$ satisfy that 
\begin{align*}
 0 \leq \phi_1 \leq 1,
\qquad
\text{Supp}(\phi_1) \subseteq [2, 8],
\qquad
\phi_1 (r) = 1 \,\, \text{for} \,\, 3 \leq r \leq 7.
\end{align*}
For any $k \geq 2$, let 
\begin{align*}
 \phi_k(r) = \phi_1(2^{-k+1} r). 
\end{align*}
Then $\text{Supp}(\phi_k) \subseteq [2^k, 2^{k+2}]$. Moreover, 
\begin{align*}
  \sum_{k=0}^\infty \phi_k(r)
> 0, 
\qquad 
\forall \, r \in [0, \infty).
\end{align*}
Define
\begin{align*}
 \theta_k(v)
= \frac{\phi_k(|v|)}{\sum_{j=0}^\infty \phi_j(|v|)},
\qquad  
 \forall \, v \in \R^3. 
\end{align*}
Then $\theta_k$ is well-defined and the first three conditions in~\eqref{cond:theta-k-1} hold. Note that
\begin{align*}
  \sum_{k=0}^\infty \theta_k^2(v)
&= \theta_{k-1}^2(v) + \theta_k^2(v)
  + \theta_{k+1}^2(v)
\geq 
 \frac{1}{3} \vpran{\theta_{k-1}(v) + \theta_k(v)
  + \theta_{k+1}(v)}^2
\\
&= \frac{1}{3} \vpran{\sum_{k=0}^\infty \theta_k(v)}^2
= \frac{1}{3}.
\end{align*}
Therefore, the last condition in~\eqref{cond:theta-k-1} holds. 

To bound the derivatives, we note that for any $v \in \text{Supp}(\theta_k)$,
\begin{align*}
 \theta_k (v)
&= \frac{\phi_k(|v|)}{\phi_{k-1}(|v|) + \phi_k(|v|) + \phi_{k+1}(|v|)}
\\
& = \frac{\phi_1(2^{-k+1}|v|)}{\phi_1(2^{-k+2}|v|) + \phi_1(2^{-k+1}|v|) + \phi_1(2^{-k}|v|)}
= h(2^{-k+1} |v|),
\end{align*}
where for $k \geq 1$,
\begin{align*}
 h(y) = \frac{\phi_1(y)}{\phi_1(2y) + \phi_1(y) + \phi_1(y/2)},
\qquad
  2 \leq y \leq 8,
\end{align*}
and for $k=0$, 
\begin{align*}
 \theta_0(v) = \frac{\phi_0(|v|)}{\phi_0(|v|) + \phi_1 (|v|)}, 
\qquad |v| \in [0, 4].
\end{align*}
Therefore, $|\theta_0'| \leq C$ and for $k \geq 1$,
\begin{align*}
 \nabla_v \theta_k
= h'(2^{-k} |v|) 2^{-k} \frac{v}{|v|},  
\qquad  
 \nabla_v^2 \theta_k
= h''(2^{-k} |v|) 2^{-2k}
  \frac{v}{|v|} \otimes 
  \frac{v}{|v|}
  + h'(2^{-k} |v|) 2^{-k}
    \nabla_v \frac{v}{|v|},
\end{align*}
where $|v| \in [2^{k},  2^{k+2}]$. 
By the boundedness of $h', h''$, we obtain the desired bounds in~\eqref{cond:theta-k-2}. Moreover, due to the compact support of $h$, we have
\begin{equation*}
 \sum_{k=1}^\infty |h'(2^{-k} |v|)| < 2 \norm{h'}_{L^\infty(2,8)}. 
\qedhere
\end{equation*}
\end{proof}

The final technical lemma in this section concerns the classical uniform convergence of regularized functions via convolution. 

\begin{lemma} \label{lem:kernels}
Suppose $h \in C([0, T_0] \times \T^3)$ and $g \in C([0, T_0] \times \T^3 \times \R^3) \cap L^\infty([0, T_0] \times \T^3; L^\infty_{k_0}(\R^3))$ with $k_0 > 0$. 

\Ni (a) Suppose $h \geq 0$. Then there exists $\phi_a(x) \geq 0$ such that 
\begin{align*}
 h_a(t, x) := h * \phi_a \geq 0, 
\qquad
 h_a \to h \,\, \text{in} \,\,
C([0, T_0] \times \T^3)
\,\,\,
\text{as $a \to 0$}.
\end{align*}

\Ni (b) Let $\psi_a \in C^\infty_c(\T^3 \times \R^3)$ be a mollifier in $(x, v)$. Then 
\begin{align*}
 g_a (t, x, v)
:= g \ast \psi_a 
\to g \,\,\,
\text{in} \,\,\,
C([0, T_0] \times T^3 \times \R^3) \,\,\,
\text{as $a \to 0$.}
\end{align*}
\end{lemma}

\begin{proof}
(a) For positivity, we can choose $\phi_a$ as a positive summability kernel such as the Fej\'er kernel. The proof of uniform convergence over the compact set $[0, T_0] \times \T^3$ can be found in~\cite{Evans} Appendix C. 

\Ni (b) Once again, the uniform convergence over a compact set $[0, T_0] \times \T^3 \times B(0, R)$ follows from the standard argument in~\cite{Evans}. To obtain the uniform convergence on $[0, T_0] \times \T^3 \times \R^3$, we make use of the decay tail of $g$. Note that 
\begin{align*}
  |v|^{k_0} \abs{g \ast \psi_a}
&\leq 
 \int_{\R^3} |v|^{k_0} |g(v-w)| \psi_a(w) \dw
\\
&\leq
 C \int_{\R^3} |v - w|^{k_0} |g(v-w)| \psi_a(w) \dw
 + C  \int_{\R^3}  |g(v-w)| |w|^{k_0} \psi_a(w) \dw
\\
&\leq 
C\sup\vpran{|v|^{k_0} g} 
+ C \sup |g|
< \infty.
\end{align*}
Hence for any $|v| > R$, we have
\begin{align*}
 \abs{g \ast \psi_a - g} \One_{|v| > R}
\leq
 \abs{g \ast \psi_a} \One_{|v| > R}
 + \abs{g} \One_{|v| > R}
\leq  
 \frac{C}{R^{k_0}} \to 0
\qquad
\text{as $R \to \infty$,}
\end{align*}
which combined with the uniform convergence on $[0, T_0] \times \T^3 \times B(0, R)$ gives the uniform convergence on $[0, T_0] \times \T^3 \times \R^3$.
\end{proof}

\section{A Toy Model}\label{Sec:toy}
We start with a toy model to illustrate the idea. Consider the equation
\begin{align} \label{eq:toy-original}
 \del_t f + v \cdot \nabla_x f 
&= \nabla_v \cdot \vpran{\rho(t, x) \vint{v}^\beta \nabla_v f}, 
\qquad
  \rho(t, x) =  \int_{\R^3} f(t, x, v) \dv, 
\\
 f|_{t=0} & = f_0 (x, v),  \nn
\end{align}
where $\beta \leq 2$. This toy model mimics, to some extent, the Landau equation with both soft and hard potentials, with $\beta < 0$ being soft and $\beta > 0$ being hard. The key difference, when compared with the Landau equation, is that there is no weight gap in the upper and lower bounds of the diffusive term on the right-hand side. The uniqueness of solutions to this toy model suggests that whether the potential is soft or hard does not make an essential difference. It is possible that the condition $\beta \leq 2$ is technical rather than structural. 

Suppose the initial data of $f$ satisfies
\begin{align} \label{cond:toy-initial}
   0 \leq f_0 \leq m_0, 
\qquad  
  \norm{f_0 \vint{v}^{k_0}}_{L^\infty_{x} (L^\infty_v \cap L^1_v)} < \infty, 
\qquad
   \rho(0, x) > c_0 > 0, 
\end{align}
where $m_0, c_0$ are constants and $k_0 > 3 + |\beta|$. 

\begin{lemma} \label{lem:toy-space}
Suppose there exist $T_\ast > 0$ and $k_0 > \max\{3 + |\beta|, 4\}$ such that $\vint{v}^{k_0} f \in C([0, T_\ast] \times \T^3 \times \R^3)$ is a continuous solution to~\eqref{eq:toy-original} with 
initial data satisfying~\eqref{cond:toy-initial} and 
\begin{align*}
  \norm{f \vint{v}^{k_0}}_{L^\infty_{t,x} (L^\infty_v \cap L^1_v)} < \infty.
\end{align*}
Then $\rho$ is uniformly continuous  on $[0, T] \times \T^3$ and for $T_0$ small enough, 
\begin{align} \label{bound:rho-lower}
   \rho(t, x) \geq \frac{c_0}{2}, 
\qquad
   \forall (t, x) \in [0, T_0] \times \T^3.
\end{align}
\end{lemma}
\begin{proof}
The uniform continuity of $\rho$ on $[0, T] \times \T^3$ follows from the Lebesgue Dominated Convergence theorem, given that $\vint{v}^4 f \in C([0, T_\ast] \times \T^3 \times \R^3)$. The bound~\eqref{bound:rho-lower} follows directly from continuity. 
\end{proof}


A few words are in order to illustrate the main ideas. First, ideally, we want to commute the $\CalM$ operator with $\rho \vint{v}^\beta$. However, this commutation requires high regularity of $\rho$ in $x$ which is not known a priori. To overcome this difficulty, we decompose $\rho$ into its regular and singular parts through convolution. 

Second, equation~\eqref{trasp_comm} shows that the commutator $\CalM$ with the transport operator introduces a backward diffusion with no weights in $v$. Hence, in the case of soft potentials where $\beta < 0$, the regularity gained through diffusion is not sufficient to control the transport commutator. To address this, we incorporate the weight into the $\CalM$ operator through the dyadic decomposition. It turns out that such a ``weighted" $\CalM$ is also needed in the hard potential case where $\beta > 0$. 

Let us now state the uniqueness theorem.
\begin{theorem}  \label{thm:unique-w-toy}
Suppose $\beta \leq 2$ and $f, g$ are two continuous solutions to~\eqref{eq:toy-original} satisfying the conditions in Lemma~\ref{lem:toy-space}. 
Then $f = g$. 
\end{theorem}

\begin{proof}
First, we use the dyadic decomposition to localize the velocity variable. Let $\theta_n(v)$ be the partition of unity in Lemma~\ref{lem:dyadic}. Let
\begin{align} \label{def:differences}
 h = f - g, 
\qquad 
 w_n = \theta_n\vint{v}^m (f - g), 
\quad
 3 < m \leq k_0,
\quad 
 m + \beta \leq k_0.
\end{align}
Then $\text{Supp}(w_n) \subseteq B(0, 2^{n+2})\setminus \overline{B(0, 2^n)}$ and $w_n$ satisfies
\begin{align} \label{eq:w-toy}
  \del_t w_n + v \cdot \nabla_x w_n 
&= \nabla_v \cdot \vpran{\rho_g \vint{v}^\beta \nabla_v w_n}
 + \nabla_v \cdot \vpran{\rho_h \theta_n \vint{v}^{m+\beta} \nabla_v f} \nn
\\
& \quad \,
 - \nabla_v \vpran{\theta_n \vint{v}^m}    \cdot \rho_h \vint{v}^\beta
   \nabla_v f
 - \nabla_v\vpran{\theta_n \vint{v}^m}    \cdot \rho_g \vint{v}^\beta
   \nabla_v h \nn
\\
& \quad \,
  - \nabla_v \cdot 
    \vpran{\rho_g \vint{v}^\beta
    \nabla_v\vpran{\theta_n \vint{v}^m} h},
\\
 w_n|_{t=0} & = 0,  \nn
\end{align}
where $\rho_w(t, x) =  \int_{\R^3} w(t, x, v) \dv$. Let $\phi_a$ be a non-negative kernel in Lemma~\ref{lem:kernels} (a) and decompose $\rho_g$ into its regular and singular parts
\begin{align*}
 \rho_g = \rho_g \ast \phi_a + (\rho_g - \rho_g \ast \phi_a). 
\end{align*}
Rewrite equation~\eqref{eq:w-toy} accordingly: 
\begin{align} \label{eq:w-toy-reg}
  \del_t w_n + v \cdot \nabla_x w_n 
&= \nabla_v \cdot \vpran{\vpran{\rho_g \ast \phi_a} \vint{v}^\beta \nabla_v w_n}
  + \nabla_v \cdot \vpran{\vpran{\rho_g - \rho_g \ast \phi_a} \vint{v}^\beta \nabla_v w_n} \nn
\\
& \quad \,
 + \nabla_v \cdot \vpran{\rho_h \theta_n \vint{v}^{m+\beta} \nabla_v f}
 - \nabla_v \cdot 
    \vpran{\rho_g \vint{v}^\beta
    \nabla_v\vpran{\theta_n \vint{v}^m} h} 
\\
& \quad \,
 - \nabla_v \vpran{\theta_n \vint{ v}^m}    \cdot \rho_h \vint{v}^\beta
   \nabla_v f
 - \nabla_v\vpran{\theta_n \vint{v}^m}    \cdot \rho_g \vint{v}^\beta
   \nabla_v h := \sum_{i=1}^6 J_i, \nn
\end{align}
where $a > 0$ is chosen to be small enough such that 
\begin{align} \label{cond:ellip}
 \rho_a \ast \phi_a > c_0/4 > 0.
\end{align}
This is possible due to the uniform convergence in Lemma~\ref{lem:kernels} (a). 

Instead of using the generic $\CalM$-operator, we define the weighted version to accommodate the extra weight $\vint{v}^\beta$: for each $n \in \N \cup \{0\}$, define
\begin{align*}
  M_n
:= \frac{1}{1 + \delta 2^{\beta n} \int_0^{T-t} \langle \xi - \tau \eta\rangle^{2} \;d\tau}.
\end{align*}
Denote $M_n$ as the symbol of the operator $\CalM_n$ so that
\begin{align*}
 \CalF\vpran{\CalM_n w_n}
=  M_n \widehat{w_n}.
\end{align*}
Apply the $\CalM_n$-operator to~\eqref{eq:w-toy-reg} and we obtain 
\begin{align*}
  \del_t \CalM_n w_n + v \cdot \nabla_x \CalM_n w_n
= \sum_{i=1}^6 \CalM_n J_i + R_1,
\end{align*}
where $R_1$ is the transport commutator defined in~\eqref{trasp_comm}
\begin{align} \label{def:R-1}
\widehat{R_1} = \CalF\vpran{[\CalM_n, \del_t + v \cdot \nabla_x] w_n}
= \delta 2^{\beta n} \tfrac{\langle \xi \rangle^{2}}{1 + \delta 2^{\beta n} \int_t ^T \langle \xi +(t- \tau) \eta\rangle^{2} \;d\tau} \CalF\vpran{\CalM_n w_n}.
\end{align}
Perform an energy estimate for each $\CalM_n w_n$. We have that
\begin{align*}
   \int_0^{T_0} \int_0^T \iint_{\T^3 \times \R^3}
   \CalM_n w_n \, \del_t \CalM_n w_n
= \frac{1}{2} \int_0^{T_0}      
   \iint_{\T^3 \times \R^3}
|w_n|^2 \dx\dv\dt,
\end{align*}
\begin{align*}
 \int_0^{T_0} \int_0^T \iint_{\T^3 \times \R^3}
   \CalM_n w_n \, \vpran{v \cdot \nabla_x \CalM_n w_n} = 0.
\end{align*}
Summing over $n \in \N$ and applying the last inequality in~\eqref{cond:theta-k-1}, we have
\begin{align} \label{est:LHS}
\sum_{n=0}^\infty
\int_0^{T_0} \int_0^T \iint_{\T^3 \times \R^3}
   \CalM_n w_n \, \del_t \CalM_n w_n
\geq \frac{1}{6}
  \norm{w}_{L^2_{t,x,v}}^2.
\end{align}
Now we estimate the contribution to the energy estimate of each $J_i$, $i = 1, 2, \cdots, 6$. For the ease of notation, we name the contribution of $R_1$ and each $J_i$ in the energy estimate as $ER_1$ and $EJ_i$ respectively. An important observation is that the bounding coefficients $C$ in Lemmas~\ref{lem:M-comm-x}-\ref{lem:M-comm-x-v} are independent of $\delta$. Thus when these lemmas are applied to $\CalM_n$, the resulting coefficients in the upper bounds are independent of $n. $

\smallskip

\Ni \underline{Estimate of $EJ_1$:} The $EJ_1$-term for each $n$ has the form
\begin{align*}
& \quad \,
 \int_0^{T_0} \int_0^T \iint_{\T^3 \times \R^3}
   \vpran{\CalM_n w_n} \,
\CalM_n \nabla_v \cdot \vpran{\vpran{\rho_g \ast \phi_a} \vint{v}^\beta \nabla_v w_n}
\\
& = - \int_0^{T_0} \int_0^T \iint_{\T^3 \times \R^3}
   \vpran{\nabla_v \CalM_n w_n} \cdot
\CalM_n \vpran{\vpran{\rho_g \ast \phi_a} \vint{v}^\beta \nabla_v w_n}. 
\end{align*}
Commutating $\CalM_n$ with $\rho_g \ast \phi_a$, we have
\begin{align*}
  \CalM_n \vpran{\vpran{\rho_g \ast \phi_a} \vint{v}^\beta \nabla_v w_n}
&= \vpran{\rho_g \ast \phi_a} \vint{v}^\beta \CalM_n \nabla_v w_n
+ [\CalM_n, \,\, \rho_g \ast \phi_a]
  \vint{v}^\beta \nabla_v w_n
\\
& \quad \,
+ \vpran{\rho_g \ast \phi_a} [\CalM_n, \,\, \vint{v}^\beta] \nabla_v w_n. 
\end{align*}
By Lemma~\ref{lem:M-comm-x}, we have 
\begin{align*}
 \norm{[\CalM_n, \,\, \rho_g \ast \phi_a]
  \vint{v}^\beta \nabla_v w_n}_{L^2_{x, v}}
&\leq
 C T \norm{\rho_g \ast \phi_a}_{H^5}
 \norm{\CalM_n \vpran{\vint{v}^\beta \nabla_v w_n}}_{L^2_{x, v}}
\\
&\leq
 \frac{C T}{a^5} \norm{\rho_g}_{L^\infty} 
 \norm{\CalM_n \vpran{\vint{v}^\beta \nabla_v w_n}}_{L^2_{x, v}}, 
\end{align*}
where by Lemma~\ref{lem:M-comm-v} and by the support of $w_n$,
\begin{align*}
 \norm{\CalM_n \vpran{\vint{v}^\beta \nabla_v w_n}}_{L^2_{x, v}}
&\leq 
 \norm{\vint{v}^\beta \CalM_n \vpran{\nabla_v w_n}}_{L^2_{x, v}}
+  \norm{[\CalM_n, \,\, \vint{v}^\beta] \nabla_v w_n}_{L^2_{x, v}}
\\
&\leq 
 C\, 2^{\beta n} 
 \norm{\CalM_n \vpran{\nabla_v w_n}}_{L^2_{x, v}}\,.
\end{align*}
Therefore,
\begin{align} \label{est:toy-comm-J-1-1}
 \norm{[\CalM_n, \,\, \rho_g \ast \phi_a]
  \vint{v}^\beta \nabla_v w_n}_{L^2_{x, v}}
\leq
 2^{\beta n} 
 \frac{C T}{a^5} \norm{\rho_g}_{L^\infty} 
 \norm{\CalM_n \vpran{\nabla_v w_n}}_{L^2_{x, v}}.
\end{align}
Similarly, by Lemma~\ref{lem:M-comm-v}, 
\begin{align} \label{est:toy-comm-J-1-2}
 \norm{\vpran{\rho_g \ast \phi_a} [\CalM_n, \,\, \vint{v}^\beta] \nabla_v w_n}_{L^2_{x,v}}
&\leq 
 \norm{\vpran{\rho_g \ast \phi_a}}_{L^\infty}
 \norm{[\CalM_n, \,\, \vint{v}^\beta] \nabla_v w_n}_{L^2_{x,v}}
\\
&\leq
 C 2^{\frac{\beta n}{2}} \norm{\rho_g}_{L^\infty}
 \norm{\vint{v}^{\frac{\beta}{2}-1}w_n}_{L^2_{x,v}}
\leq
  C 2^{\frac{\beta n}{2}} \norm{\rho_g}_{L^\infty}
\norm{w_n}_{L^2_{x,v}}, \nn
\end{align}
where the last inequality is one place where the technical assumption of $\beta \leq 2$ is used. The factor $2^{\frac{\beta n}{2}}$ will later be combined with the other term $\nabla_v \CalM_n w_n$.

Estimates~\eqref{est:toy-comm-J-1-1} and~\eqref{est:toy-comm-J-1-2} give
\begin{align} \label{est:toy-J-1}
& \quad \,
 - \int_0^{T_0} \int_0^T \iint_{\T^3 \times \R^3}
   \vpran{\nabla_v \CalM_n w_n} \cdot
\CalM_n \vpran{\vpran{\rho_g \ast \phi_a} \vint{v}^\beta \nabla_v w_n} \nn
\\
& \leq 
 - 4^{-\min(\beta, 0)} \frac{c_0}{8} 2^{\beta n}
   \int_0^{T_0} \int_0^T \iint_{\T^3 \times \R^3}
   \abs{\nabla_v \CalM_n w_n}^2 \dv \dx \dt \dT \nn
\\
& \quad \,
 + 2^{\beta n} 
 \frac{C T}{a^5} \norm{\CalM_n \vpran{\nabla_v w_n}}_{L^2_{x, v}}^2
 + C T \norm{\rho_g}_{L^\infty}^2
 \norm{w_n}^2_{L^2_{x,v}} \nn
\\
& \leq 
 \vpran{- 4^{-\min(\beta, 0)} \frac{c_0}{8} 
 + \frac{C T_0}{a^5}}
   \int_0^{T_0} \int_0^T \iint_{\T^3 \times \R^3}
   \abs{2^{\frac{\beta n}{2}}\nabla_v \CalM_n w_n}^2 \dv \dx \dt \dT \nn
\\
& \quad \,
+ C T_0 \norm{\rho_g}_{L^\infty}^2
 \norm{w_n}^2_{L^2_{t,x,v}}, 
\end{align}
by taking $T_0$ small enough. The size of $T_0$ depends on $\norm{\rho_g}_{L^\infty([0, T_\ast] \times \T^3)}$. Summing over $n \in \N$, we have
\begin{align} \label{est:EJ-1}
 EJ_1
\leq
\vpran{- 4^{-\min(\beta, 0)} \frac{c_0}{8}
+ \frac{C T}{a^5}} \sum_{n=0}^\infty
 \norm{2^{\frac{\beta n}{2}} \nabla_v \CalM_n w_n}_{L^2_{T,t,x,v}}^2
+ C T \norm{\rho_g}_{L^\infty}
 \norm{w}^2_{L^2_{t,x,v}}. 
\end{align}

\smallskip

\Ni \underline{Estimate of $EJ_2$:} The $EJ_2$-term for each $n$ has the form
\begin{align*}
& \quad \,
 \int_0^{T_0} \int_0^T \iint_{\T^3 \times \R^3}
   \vpran{\CalM_n w_n} \,
\CalM_n \nabla_v \cdot \vpran{\vpran{\rho_g - \rho_g \ast \phi_a} \vint{v}^\beta \nabla_v w_n}
\\
& = - \int_0^{T_0} \int_0^T \iint_{\T^3 \times \R^3}
   \vpran{\nabla_v\CalM_n w_n} \,
\cdot \nabla_v \CalM_n \vpran{\vpran{\rho_g - \rho_g \ast \phi_a} \vint{v}^\beta w_n}
\\
& \quad \,
+ \int_0^{T_0} \int_0^T \iint_{\T^3 \times \R^3}
   \vpran{\nabla_v\CalM_n w_n} \,
\cdot \CalM_n \vpran{\vpran{\rho_g - \rho_g \ast \phi_a} \vpran{\nabla_v\vint{v}^\beta} w_n}
=: J_{2, 1} + J_{2,2}.
\end{align*}
By Lemma~\ref{lem:M-basic}, 
\begin{align*}
 \abs{J_{2, 1}}
& \leq 
 \norm{\rho_g - \rho_g \ast \phi_a}_{L^\infty_{t,x}}
 \norm{2^{\frac{\beta n}{2}}\nabla_v \CalM_n w_n}_{L^2_{T,t,x,v}}
 \norm{\nabla_v \CalM_n \vpran{2^{-\frac{\beta n}{2}}\vint{v}^\beta w_n}}_{L^2_{T,t,x,v}}
\\
& \leq 
 \frac{C}{\sqrt{\delta 2^{\beta n}}}\norm{\rho_g - \rho_g \ast \phi_a}_{L^\infty_{t,x}}
 \norm{2^{\frac{\beta n}{2}}\nabla_v \CalM_n w_n}_{L^2_{T,t,x,v}} 
 \norm{2^{-\frac{\beta n}{2}}\vint{v}^\beta w_n}_{L^2_{t,x,v}}
\\
& \leq 
 \frac{C}{\sqrt{\delta}}\norm{\rho_g - \rho_g \ast \phi_a}_{L^\infty_{t,x}}
 \norm{2^{\frac{\beta n}{2}}\nabla_v \CalM_n w_n}_{L^2_{T,t,x,v}} 
 \norm{w_n}_{L^2_{t,x,v}}. 
\end{align*}
Note that the last step is one reason why $\CalM_n$ is designed by replacing $\delta$ with the weighted version of $\delta_n$. Similarly, 
\begin{align*}
 \abs{J_{2,2}}
&\leq
 C \norm{\rho_g - \rho_g \ast \phi_a}_{L^\infty_{t,x}}
 \norm{2^{\frac{\beta n}{2}}\nabla_v \CalM_n w_n}_{L^2_{T,t,x,v}}
 \norm{\vint{v}^{\frac{\beta}{2}-1} w_n}_{L^2_{T,t,x,v}}
\\
&\leq
 C \norm{\rho_g - \rho_g \ast \phi_a}_{L^\infty_{t,x}}
 \norm{2^{\frac{\beta n}{2}}\nabla_v \CalM_n w_n}_{L^2_{T,t,x,v}}
 \norm{w_n}_{L^2_{t,x,v}} 
\end{align*}
for $\beta \leq 2$. Therefore, the estimate related to $J_2$ is
\begin{align} \label{est:J-2}
& \quad \,
 \int_0^{T_0} \int_0^T \iint_{\T^3 \times \R^3}
   \vpran{\CalM_n w_n} \,
\CalM_n \nabla_v \cdot \vpran{\vpran{\rho_g - \rho_g \ast \phi_a} \vint{v}^\beta \nabla_v w_n} \nn
\\
& \leq 
  \frac{C}{\sqrt{\delta}}\norm{\rho_g - \rho_g \ast \phi_a}_{L^\infty_{t,x}}
 \norm{2^{\frac{\beta n}{2}}\nabla_v \CalM_n w_n}_{L^2_{T,t,x,v}}
 \norm{w_n}_{L^2_{t,x,v}},
\qquad
 0 < \delta < 1.
\end{align}
Denote $\Eps_a
:= \norm{\rho_g - \rho_g \ast \phi_a}_{L^\infty_{t,x}}$.
Then
\begin{align*}
 \Eps_a \to 0
\quad \text{as} \,\,
a \to 0.
\end{align*}
Summing over $n \in \N$, we have
\begin{align} \label{est:EJ-2}
 EJ_2
\leq 
 \Eps_a \sum_{n=0}^\infty
 \norm{2^{\frac{\beta n}{2}} \nabla_v \CalM_n w_n}_{L^2_{T,t,x,v}}^2
+ \frac{C \Eps_a}{\delta}
  \norm{w}_{L^2_{t,x,v}}^2.
\end{align}


\Ni \underline{Estimate of $EJ_3$:} The $EJ_3$-term for each $n$ has the form
\begin{align*}
& \quad \,
 \int_0^{T_0} \int_0^T \iint_{\T^3 \times \R^3}
   \vpran{\CalM_n w_n} \,
   \CalM_n \nabla_v \cdot \vpran{\rho_h \theta_n \vint{v}^{m+\beta} \nabla_v f}
\\
& = - \int_0^{T_0} \int_0^T \iint_{\T^3 \times \R^3}
   \vpran{\nabla_v \CalM_n w_n} \, \cdot
   \CalM_n \vpran{\rho_h \theta_n \vint{v}^{m+\beta} \nabla_v \vpran{f \ast \psi_a}}
\\
& \quad \,
  - \int_0^{T_0} \int_0^T \iint_{\T^3 \times \R^3}
   \vpran{\nabla_v \CalM_n w_n} \, \cdot
   \CalM_n \vpran{\rho_h \theta_n \vint{v}^{m+\beta} \nabla_v \vpran{f - f \ast \psi_a}}
= : J_{3,1} + J_{3,2},
\end{align*}
where $\psi_a$ is the mollifier in $x, v$. By Lemma~\ref{lem:M-basic}, 
\begin{align*}
 \abs{J_{3,1}}
&\leq
 C \norm{\rho_h}_{L^2_{t,x}}
 \norm{2^{\frac{\beta n}{2}}\nabla_v \CalM_n w_n}_{L^2_{T,t,x,v}}
 \norm{2^{-\frac{\beta n}{2}} \theta_n \vint{v}^{m+\beta} \vpran{f \ast \nabla_v\psi_a}}_{L^\infty_{T,t,x}L^2_v}
\\
& \leq
 C \sqrt{T}
 \norm{\rho_h}_{L^2_{t,x}}
 \norm{2^{\frac{\beta n}{2}}\nabla_v \CalM_n w_n}_{L^2_{T,t,x,v}}
 \norm{\theta_n \vint{v}^{m+\frac{\beta}{2}} \vpran{f \ast \nabla_v\psi_a}}_{L^\infty_{t,x}(L^\infty_v \cap L^1_v)}
\\
& \leq
 C \sqrt{T}
 \norm{w}_{L^2_{t,x}}
 \norm{2^{\frac{\beta n}{2}}\nabla_v \CalM_n w_n}_{L^2_{T,t,x,v}}
 \norm{\theta_n \vint{v}^{m+\frac{\beta}{2}} \vpran{f \ast \nabla_v\psi_a}}_{L^\infty_{t,x}(L^\infty_v \cap L^1_v)},
\end{align*}
where the last step holds if $m > 3$. 
The term $J_{3,2}$ is bounded as follows
\begin{align*}
 \abs{J_{3,2}}
& \leq
\abs{\int_0^{T_0} \int_0^T \iint_{\T^3 \times \R^3}
   \vpran{\nabla_v \CalM_n w_n} \, \cdot
   \nabla_v \CalM_n \vpran{\rho_h \theta_n \vint{v}^{m+\beta} \vpran{f - f \ast \psi_a}}}
\\
& \quad \,
 + \abs{\int_0^{T_0} \int_0^T \iint_{\T^3 \times \R^3}
   \vpran{\nabla_v \CalM_n w_n} \, \cdot
   \CalM_n \vpran{\rho_h \nabla_v \vpran{\theta_n \vint{v}^{m+\beta}} \vpran{f - f \ast \psi_a}}}
=: J^{(1)}_{3,2} + J^{(2)}_{3,2},
\end{align*}
where
\begin{align*}
 J^{(1)}_{3,2}
& \leq
 2^{-\frac{\beta n} {2}}
 \norm{2^{\frac{\beta n}{2}}\nabla_v \CalM_n w_n}_{L^2_{T,t,x,v}}
 \norm{\nabla_v \CalM_n \vpran{\rho_h \theta_n
 \vint{v}^{m+\beta} \vpran{f - f \ast \psi_a}}}_{L^2_{T,t,x,v}}
\\
& \leq
 \frac{C 2^{-\frac{3\beta n}{2}}}{\sqrt{\delta}}
 \norm{\rho_h}_{L^2_{t,x}}
 \norm{2^{\frac{\beta n}{2}}\nabla_v \CalM_n w_n}_{L^2_{T,t,x,v}} 
 \norm{ \theta_n
 \vint{v}^{m+\beta} \vpran{f - f \ast \psi_a}}_{L^\infty_{t,x}L^2_v}
\\
& \leq
 \frac{C}{\sqrt{\delta}}
 \norm{\rho_h}_{L^2_{t,x}}
 \norm{2^{\frac{\beta n}{2}}\nabla_v \CalM_n w_n}_{L^2_{T,t,x,v}} 
 \norm{ \theta_n
 \vint{v}^{m-\frac{\beta}{2}} \vpran{f - f \ast \psi_a}}_{L^\infty_{t,x}L^2_v}
\\
& \leq
  \frac{C}{\sqrt{\delta}}
 \norm{w}_{L^2_{t,x}}
 \norm{2^{\frac{\beta n}{2}}\nabla_v \CalM_n w_n}_{L^2_{T,t,x,v}} 
 \norm{\theta_n
 \vint{v}^{m-\frac{\beta}{2}} \vpran{f - f \ast \psi_a}}_{L^\infty_{t,x}(L^\infty_v \cap L^1_v)},
\end{align*}
if $m > 3$. Moreover, for the same $m$,
\begin{align*}
 J^{(2)}_{3,2}
&\leq
 2^{-\frac{\beta n} {2}}
 \norm{2^{\frac{\beta n}{2}}\CalM_n w_n}_{L^2_{T,t,x,v}}
 \norm{\CalM_n \vpran{\rho_h \nabla_v \vpran{\theta_n
 \vint{v}^{m+\beta}} \vpran{f - f \ast \psi_a}}}_{L^2_{T,t,x,v}}
\\
& \leq
  C \sqrt{T} \norm{\rho_h}_{L^2_{t,x}}
  2^{-\frac{\beta n} {2}}
 \norm{2^{\frac{\beta n}{2}}\CalM_n w_n}_{L^2_{T,t,x,v}}
 \norm{\nabla_v \vpran{\theta_n
 \vint{v}^{m+\beta}} \vpran{f - f \ast \psi_a}}_{L^\infty_{t,x}L^2_v}
\\
& \leq 
 C \norm{w}_{L^2_{t,x}}
 \norm{2^{\frac{\beta n}{2}}\CalM_n w_n}_{L^2_{T,t,x,v}}
 \norm{\nabla_v \vpran{\theta_n
 \vint{v}^{m+\frac{\beta}{2}}} \vpran{f - f \ast \psi_a}}_{L^\infty_{t,x}(L^\infty_v \cap L^1_v)}. 
\end{align*}
Altogether, the estimate related to $J_3$ is
\begin{align}
& \quad \,
 \int_0^{T_0} \int_0^T \iint_{\T^3 \times \R^3}
  \vpran{\CalM_n w_n} \,
   \CalM_n \nabla_v \cdot \vpran{\rho_h \theta_n \vint{v}^{m+\beta} \nabla_v f}  \nn
\\
& \leq
 C \norm{w}_{L^2_{t,x}}
 \norm{2^{\frac{\beta n}{2}}\CalM_n w_n}_{L^2_{T,t,x,v}} \times \nn
\\
& \quad \, \times
\Big(\sqrt{T} \norm{\theta_n \vint{v}^{m+\frac{\beta}{2}} \vpran{f \ast \nabla_v\psi_a}}_{L^\infty_{t,x}(L^\infty_v \cap L^1_v) } + \frac{1}{\sqrt{\delta}}\norm{\theta_n
 \vint{v}^{m-\frac{\beta}{2}} \vpran{f - f \ast \psi_a}}_{L^\infty_{t,x}(L^\infty_v \cap L^1_v)} \nn
\\
& \hspace{2cm}
+ \norm{\nabla_v \vpran{\theta_n
 \vint{v}^{m+\frac{\beta}{2}}} \vpran{f - f \ast \psi_a}}_{L^\infty_{t,x}(L^\infty_v \cap L^1_v)}
 \Big). 
\end{align}
Instead of introducing a new notation, we re-define $\Eps_a$ as
\begin{align*}
\Eps_a
= \max\left\{\norm{\rho_g - \rho_g \ast \phi_a}_{L^\infty_{t,x}}, \,\,
\norm{\vint{v}^{m+\frac{|\beta|}{2}} \vpran{f - f \ast \psi_a}}_{L^\infty_{t,x}(L^\infty_v \cap L^1_v)} \right\}. 
\end{align*}
Then by the Lebesgue Dominated Convergence theorem, $\Eps_a \to 0$ as $a \to 0$ for $m+\frac{|\beta|}{2} + 3 < k_0$.
Summing over $n \in \N$, we have
\begin{align*}
 \sum_{n=0}^\infty
\norm{\theta_n \vint{v}^{m+\beta} \vpran{f \ast \nabla_v\psi_a}}_{L^\infty_{t,x,v}}^2
&\leq 
 \vpran{\sum_{n=0}^\infty \theta_n}
\norm{\vint{v}^{m+\frac{\beta}{2}} \vpran{f \ast \nabla_v\psi_a}}_{L^\infty_{t,x}(L^\infty_v \cap L^1_v)}
\\
& \hspace{-2cm}
= \norm{\vint{v}^{m+\frac{\beta}{2}} \vpran{f \ast \nabla_v\psi_a}}_{L^\infty_{t,x}(L^\infty_v \cap L^1_v)}
\leq
\frac{C}{a} \norm{\vint{v}^{m+\frac{\beta}{2}} f}_{L^\infty_{t,x}(L^\infty_v \cap L^1_v)}.
\end{align*}
Similarly, 
\begin{align*}
 \sum_{n=0}^\infty
 \norm{\theta_n
 \vint{v}^{m-\frac{\beta}{2}} \vpran{f - f \ast \psi_a}}_{L^\infty_{t,x}(L^\infty_v \cap L^1_v)}^2
\leq
 C \norm{\vint{v}^{m+|\beta|} \vpran{f - f \ast \psi_a}}_{L^\infty_{t,x}(L^\infty_v \cap L^1_v)}^2
\leq C \Eps_a,
\end{align*}
and by Lemma~\ref{lem:dyadic}, 
\begin{align*}
& \quad \,
 \sum_{n=0}^\infty
\norm{\nabla_v \vpran{\theta_n
 \vint{v}^{m+\frac{\beta}{2}}} \vpran{f - f \ast \psi_a}}_{L^\infty_{t,x}(L^\infty_v \cap L^1_v)}^2
\\
& \leq
 C \sum_{n=0}^\infty
\norm{\vpran{\theta_n + |\nabla_v \theta_n|}
 \vint{v}^{m+\frac{\beta}{2}} \vpran{f - f \ast \psi_a}}_{L^\infty_{t,x}(L^\infty_v \cap L^1_v)}^2
\\
& \leq
 C \sum_{n=0}^\infty
\norm{\vpran{\theta_n + 2^{-n}}
 \vint{v}^{m+\frac{\beta}{2}} \vpran{f - f \ast \psi_a}}_{L^\infty_{t,x}(L^\infty_v \cap L^1_v)}^2
 \leq C \Eps_a. 
\end{align*}
Therefore, for any $\tau_0 > 0$, which will be chosen to be small enough to be absorbed into the diffusion term, we have
\begin{align} \label{est:EJ-3}
 EJ_3
\leq
 \tau_0 \sum_{n=0}^\infty
 \norm{2^{\frac{\beta n}{2}} \nabla_v \CalM_n w_n}_{L^2_{T,t,x,v}}^2
+ \frac{C}{\tau_0 \sqrt{\delta}}
  \vpran{\Eps_a + \frac{T}{a}} \norm{\vint{v}^{m+\beta} f}_{L^\infty_{t,x,v}}^2
\norm{w}_{L^2_{t,x,v}}^2\,.
\end{align}

\Ni \underline{Estimate of $EJ_4$:} The $EJ_4$-term for each $n$ has the form
\begin{align*}
& \quad \,
 - \int_0^{T_0} \int_0^T \iint_{\T^3 \times \R^3}
   \vpran{\CalM_n w_n} \,
  \CalM_n \nabla_v \cdot 
    \vpran{\rho_g \vint{v}^\beta
    \nabla_v\vpran{\theta_n \vint{v}^m} h} 
\\
& = \int_0^{T_0} \int_0^T \iint_{\T^3 \times \R^3}
   \vpran{2^{\frac{\beta n}{2}}\nabla_ v\CalM_n w_n} \, \cdot
  \CalM_n 
    \vpran{\rho_g 2^{-\frac{\beta n}{2}}\vint{v}^\beta
    \nabla_v\vpran{\theta_n \vint{v}^m} h} 
\\
& \leq 
 \norm{2^{\frac{\beta n}{2}}\CalM_n w_n}_{L^2_{T,t,x,v}}
 \norm{\CalM_n 
    \vpran{\rho_g 2^{-\frac{\beta n}{2}}\vint{v}^{\beta-m}
    \nabla_v\vpran{\theta_n \vint{v}^m} w}}_{L^2_{T,t,x,v}}
\\
& \leq
 C \sqrt{T} 
 \norm{\rho_g}_{L^\infty_{t,x}}
 \norm{2^{\frac{\beta n}{2}}\CalM_n w_n}_{L^2_{T,t,x,v}}
 \norm{\vint{v}^{\frac{\beta}{2}-1} (\theta_n + |\theta'_n|)
    w}_{L^2_{T,t,x,v}}
\\
& \leq
 C \sqrt{T} 
 \norm{\rho_g}_{L^\infty_{t,x}}
 \norm{2^{\frac{\beta n}{2}}\CalM_n w_n}_{L^2_{t,x,v}}
 \norm{(\theta_n + |\theta'_n|)
    w}_{L^2_{t,x,v}}, 
\qquad \text{for \,} \beta \leq 2.
\end{align*}
Similar as the estimate for $EJ_3$, if we sum over $n \in \N$, then
\begin{align} \label{est:EJ-4}
 EJ_4
\leq
 \sqrt{T} \sum_{n=0}^\infty
 \norm{2^{\frac{\beta n}{2}} \nabla_v \CalM_n w_n}_{L^2_{T,t,x,v}}^2
 + C \, \sqrt{T} \norm{\rho_g}_{L^\infty_{t,x}}^2
 \norm{w}_{L^2_{t,x,v}}^2,
\end{align}
where by the compact support on each ring and ~\eqref{cond:theta-k-2} in Lemma~\ref{lem:dyadic}, we have applied 
\begin{align*}
 \sum_{n=0}^\infty
 \vint{v} \abs{\nabla_v \theta_n} 
\leq 
 2 \sup_n \norm{\vint{v}\nabla_v \theta_n}_{\R^3} 
< \infty.
\end{align*}

\Ni \underline{Estimate of $EJ_5$:} The $EJ_5$-term for each $n$ has the form
\begin{align*}
& \quad \,
 - \int_0^{T_0} \int_0^T \iint_{\T^3 \times \R^3}
   \vpran{\CalM_n w_n}
   \CalM_n \vpran{\nabla_v \vpran{\theta_n \vint{v}^m}    \cdot \rho_h \vint{v}^\beta
   \nabla_v f}
\\
& = - \int_0^{T_0} \int_0^T \iint_{\T^3 \times \R^3}
   \vpran{\CalM_n w_n}
   \nabla_v \cdot \CalM_n \vpran{\nabla_v \vpran{\theta_n \vint{v}^m}    \rho_h \vint{v}^\beta
   f}
\\
& \quad \,
   +  \int_0^{T_0} \int_0^T \iint_{\T^3 \times \R^3}
   \vpran{\CalM_n w_n}
   \CalM_n \vpran{\rho_h 
   f \, \nabla_v \cdot \vpran{\nabla_v \vpran{\theta_n \vint{v}^m}    \vint{v}^\beta}},
\end{align*}
where
\begin{align*}
& \quad \,
 \abs{\int_0^{T_0} \int_0^T \iint_{\T^3 \times \R^3}
   \vpran{\CalM_n w_n}
   \nabla_v \cdot \CalM_n \vpran{\nabla_v \vpran{\theta_n \vint{v}^m}    \rho_h \vint{v}^\beta
   f}}
\\
&= \abs{\int_0^{T_0} \int_0^T \iint_{\T^3 \times \R^3}
   \vpran{\nabla_v \CalM_n w_n}
   \cdot \CalM_n \vpran{\nabla_v \vpran{\theta_n \vint{v}^m}    \rho_h \vint{v}^\beta
   f}}
\\
& \leq 
 C \sqrt{T}
 \norm{2^{\frac{\beta n}{2}}\CalM_n w_n}_{L^2_{T,t,x,v}}
 \norm{\rho_h}_{L^2_{t,x}}
 \norm{\nabla_v \vpran{\theta_n \vint{v}^m} \vint{v}^\beta f}_{L^\infty_{t,x}L^2_v}
\\
& \leq 
 C \sqrt{T}
 \norm{2^{\frac{\beta n}{2}}\CalM_n w_n}_{L^2_{T,t,x,v}}
 \norm{w}_{L^2_{t,x,v}}
 \norm{\nabla_v \vpran{\theta_n \vint{v}^m} \vint{v}^\beta f}_{L^\infty_{t,x}(L^\infty_v \cap L^1_v)}
\end{align*}
and
\begin{align*}
& \quad \,
\abs{\int_0^{T_0} \int_0^T \iint_{\T^3 \times \R^3}
   \vpran{\CalM_n w_n}
   \CalM_n \vpran{\rho_h 
   f \, \nabla_v \cdot \vpran{\nabla_v \vpran{\theta_n \vint{v}^m}    \vint{v}^\beta}}}
\\
& \leq 
\norm{\CalM_n w_n}_{L^2_{T,t,x,v}}
\norm{\rho_h}_{L^2_{t,x}}
\norm{f \nabla_v \cdot \vpran{\nabla_v \vpran{\theta_n \vint{v}^m}    \vint{v}^\beta}}_{L^\infty_{t,x}L^2_v}
\\
& \leq  
 C \, T
 \norm{w_n}_{L^2_{t,x,v}}
 \norm{w}_{L^2_{t,x,v}}
 \norm{f \nabla_v \cdot \vpran{\nabla_v \vpran{\theta_n \vint{v}^m}    \vint{v}^\beta}}_{L^\infty_{t,x}(L^\infty_v \cap L^1_v)}. 
\end{align*}
Therefore, the estimate related to $J_5$ is
\begin{align} \label{est:J-5}
& \quad \,
 - \int_0^{T_0} \int_0^T \iint_{\T^3 \times \R^3}
   \vpran{\CalM_n w_n}
   \CalM_n \vpran{\nabla_v \vpran{\theta_n \vint{v}^m}    \cdot \rho_h \vint{v}^\beta
   \nabla_v f}  \nn
\\
& \leq
 C \sqrt{T}
 \norm{2^{\frac{\beta n}{2}}\CalM_n w_n}_{L^2_{T,t,x,v}}
 \norm{w}_{L^2_{t,x,v}}
 \norm{\nabla_v \vpran{\theta_n \vint{v}^m} \vint{v}^\beta f}_{L^\infty_{t,x}L^2_v} \nn
\\
& \quad \,
 + C \, T
 \norm{w_n}_{L^2_{t,x,v}}
 \norm{w}_{L^2_{t,x,v}}
 \norm{f \nabla_v \cdot \vpran{\nabla_v \vpran{\theta_n \vint{v}^m}    \vint{v}^\beta}}_{L^\infty_{t,x}(L^\infty_v \cap L^1_v)}. 
\end{align}
By Lemma~\ref{lem:dyadic},
\begin{align*}
 \sum_{n=0}^\infty
 \norm{\nabla_v \vpran{\theta_n \vint{v}^m} \vint{v}^\beta f}_{L^\infty_{t,x}(L^\infty_v \cap L^1_v)}^2
&\leq
 \sum_{n=0}^\infty
 \norm{\vpran{\theta_n + |\nabla_v \theta_n|} \vint{v}^{m+\beta} f}_{L^\infty_{t,x}(L^\infty_v \cap L^1_v)}^2
\\
& \leq
 C \norm{\vint{v}^{m+\beta} f}_{L^\infty_{t,x}(L^\infty_v \cap L^1_v)}^2,
\end{align*}
and
\begin{align*}
& \quad \,
 \sum_{n=0}^\infty
 \norm{f \nabla_v \cdot \vpran{\nabla_v \vpran{\theta_n \vint{v}^m}    \vint{v}^\beta}}_{L^\infty_{t,x}(L^\infty_v \cap L^1_v)}^2
\\
&\leq
 C \norm{f \vint{v}^{m+\beta} \vpran{\theta_n + |\nabla_v \theta_n| + |\nabla_v^2 \theta_n|}}_{L^\infty_{t,x}(L^\infty_v \cap L^1_v)}^2
 \leq 
 C \norm{f \vint{v}^{m+\beta}}_{L^\infty_{t,x}(L^\infty_v \cap L^1_v)}^2. 
\end{align*}
Here we require that $m + \beta \leq k_0$. Therefore,
\begin{align} \label{est:EJ-5}
 EJ_5
\leq
 \sqrt{T} \sum_{n=0}^\infty
 \norm{2^{\frac{\beta n}{2}} \nabla_v \CalM_n w_n}_{L^2_{T,t,x,v}}^2
 + C \, T 
 \norm{f \vint{v}^{m+\beta}}_{L^\infty_{t,x}(L^\infty_v \cap L^1_v)}^2
 \norm{w}_{L^2_{t,x,v}}^2. 
\end{align}

\Ni \underline{Estimate of $EJ_6$:} The $EJ_6$-term for each $n$ has the form
\begin{align*}
& \quad \,
 - \int_0^{T_0} \int_0^T \iint_{\T^3 \times \R^3}
   \vpran{\CalM_n w_n} \,
   \CalM_n \vpran{\nabla_v\vpran{\theta_n \vint{v}^m}    \cdot \rho_g \vint{v}^\beta
   \nabla_v h}
\\
& = \int_0^{T_0} \int_0^T \iint_{\T^3 \times \R^3}
   \vpran{\nabla_v \CalM_n w_n} \, \cdot
   \CalM_n \vpran{\nabla_v\vpran{\theta_n \vint{v}^m} \rho_g \vint{v}^\beta
   h}
\\
& \quad \,
 - \int_0^{T_0} \int_0^T \iint_{\T^3 \times \R^3}
   \vpran{\CalM_n w_n} \,
   \CalM_n \vpran{\nabla_v \cdot \vpran{\vint{v}^\beta \nabla_v\vpran{\theta_n \vint{v}^m}}    
   \rho_g h},
\end{align*}
where
\begin{align*}
& \quad \,
\int_0^{T_0} \int_0^T \iint_{\T^3 \times \R^3}
   \vpran{\nabla_v \CalM_n w_n} \, \cdot
   \CalM_n \vpran{\nabla_v\vpran{\theta_n \vint{v}^m} \rho_g \vint{v}^{\beta}
   h}
\\
& = \int_0^{T_0} \int_0^T \iint_{\T^3 \times \R^3}
   \vpran{2^{\frac{\beta n}{2}}\nabla_v \CalM_n w_n} \, \cdot
   \CalM_n \vpran{2^{-\frac{\beta n}{2}}\nabla_v\vpran{\theta_n \vint{v}^m} \rho_g \vint{v}^{\beta-m}
   w}
\\
& \leq 
 \norm{2^{\frac{\beta n}{2}}\CalM_n w_n}_{L^2_{T,t,x,v}}
 \norm{\rho_g}_{L^2_{t,x}}
 \norm{2^{-\frac{\beta n}{2}}\nabla_v\vpran{\theta_n \vint{v}^m} \vint{v}^{\beta-m}
   w}_{L^2_{T,t,x,v}}
\\
& \leq 
 C \sqrt{T}
 \norm{2^{\frac{\beta n}{2}}\CalM_n w_n}_{L^2_{T,t,x,v}}
 \norm{\rho_g}_{L^\infty_{t,x}}
 \norm{\nabla_v\vpran{\theta_n \vint{v}^m} \vint{v}^{\frac{\beta}{2}-m}
   w}_{L^2_{t,x,v}} ,
\end{align*}
and
\begin{align*}
& \quad \,
\abs{\int_0^{T_0} \int_0^T \iint_{\T^3 \times \R^3}
   \vpran{\CalM_n w_n} \,
   \CalM_n \vpran{\nabla_v \cdot \vpran{\vint{v}^\beta \nabla_v\vpran{\theta_n \vint{v}^m}}    
   \rho_g h}}
\\
& \leq
\norm{\CalM_n w_n}_{L^2_{T,t,x,v}}
\norm{\rho_g}_{L^\infty_{t,x}}
\norm{\nabla_v \cdot \vpran{\vint{v}^\beta \nabla_v\vpran{\theta_n \vint{v}^m}} \vint{v}^{-m} w}_{L^2_{T,t,x,v}}
\\
& \leq
 C \, T
 \norm{w_n}_{L^2_{t,x,v}}
\norm{\rho_g}_{L^\infty_{t,x}}
\norm{\nabla_v \cdot \vpran{\vint{v}^\beta \nabla_v\vpran{\theta_n \vint{v}^m}} \vint{v}^{-m} w}_{L^2_{t,x,v}}.
\end{align*}
Therefore, 
\begin{align*}
 (EJ_6)_n
&\leq
 C \sqrt{T}
 \norm{2^{\frac{\beta n}{2}}\CalM_n w_n}_{L^2_{T,t,x,v}}
 \norm{\rho_g}_{L^\infty_{t,x}}
 \norm{\nabla_v\vpran{\theta_n \vint{v}^m} \vint{v}^{\frac{\beta}{2}-m}
   w}_{L^2_{t,x,v}}
\\
& \quad \, 
 + C \, T
 \norm{w_n}_{L^2_{t,x,v}}
\norm{\rho_g}_{L^\infty_{t,x}}
\norm{\nabla_v \cdot \vpran{\vint{v}^\beta \nabla_v\vpran{\theta_n \vint{v}^m}} \vint{v}^{-m} w}_{L^2_{t,x,v}}.
\end{align*}
By Lemma~\ref{lem:dyadic}, 
\begin{align*}
 \sum_{n=0}^\infty
 \norm{\nabla_v\vpran{\theta_n \vint{v}^m} \vint{v}^{\frac{\beta}{2}-m}
   w}_{L^2_{t,x,v}}^2
& \leq
 C \sum_{n=0}^\infty
 \norm{\vpran{\theta_n \vint{v}^{\frac{\beta}{2}-1} + |\nabla_v \theta_n| \vint{v}^{\frac{\beta}{2}}}
   w}_{L^2_{t,x,v}}^2
\\
& \leq
 C \norm{\vint{v}^{\frac{\beta}{2} - 1}
   w}_{L^2_{t,x,v}}^2
\leq
 \norm{w}_{L^2_{t,x,v}}^2,
\qquad 
\text{for \,} \beta < 2.
\end{align*}
Moreover, 
\begin{align*}
& \quad \,
\sum_{n=0}^\infty
\norm{\nabla_v \cdot \vpran{\vint{v}^\beta \nabla_v\vpran{\theta_n \vint{v}^m}} \vint{v}^{-m} w}_{L^2_{t,x,v}}^2
\\
& \leq 
C \sum_{n=0}^\infty
\norm{\theta_n \vint{v}^{\beta -2} w}_{L^2_{t,x,v}}^2
+ C \sum_{n=0}^\infty
\norm{\abs{\nabla_v\theta_n} \vint{v}^{\beta -1} w}_{L^2_{t,x,v}}^2
+ C\sum_{n=0}^\infty
\norm{\abs{\nabla_v^2\theta_n} \vint{v}^{\beta} w}_{L^2_{t,x,v}}^2
\\
& \leq 
C \norm{\vint{v}^{\beta -2} w}_{L^2_{t,x,v}}^2
\leq 
C \norm{w}_{L^2_{t,x,v}}^2,
\qquad
\text{for \,} \beta \leq 2.
\end{align*}
Altogether, we have
\begin{align} \label{est:EJ-6}
 EJ_6
\leq
 \sqrt{T} 
 \sum_{n=0}^\infty
 \norm{2^{\frac{\beta n}{2}} \nabla_v \CalM_n w_n}_{L^2_{T,t,x,v}}^2
 + C \sqrt{T} \norm{\rho_g}_{L^\infty_{t,x}}^2
 \norm{w}_{L^2_{t,x,v}}^2. 
\end{align}

\Ni \underline{Estimate of $ER_1$:} The $ER_1$-term for each $n$ directly follows from~\eqref{trasp_comm} or~\eqref{def:R-1}
\begin{align} \label{est:ER-1}
 \sum_{n=0}^\infty
 \int_0^{T_0} \int_0^T \iint_{\T^3 \times \R^3}
   \vpran{\CalM_n w_n} \,
   \CalM_n R_1
\leq 
 C \delta 
\sum_{n=0}^\infty
\norm{2^{\frac{\beta n}{2}} \nabla_v \CalM_n w_n}_{L^2_{T,t,x,v}}^2.
\end{align}
Combining~\eqref{est:LHS} and the estimates for $EJ_i$'s, we have
\begin{align} \label{est:full-toy}
 \frac{1}{6} \norm{w}_{L^2_{t,x,v}}^2
&\leq
 \vpran{- c_1
 + \frac{CT}{a^5}
 + \Eps_a
 + \tau_0
 + \sqrt{T}
 + C \delta} 
 \sum_{n=0}^\infty
 \norm{2^{\frac{\beta n}{2}} \nabla_v \CalM_n w_n}_{L^2_{T,t,x,v}}^2 \nn
\\
& \quad \,
 + \vpran{C \sqrt{T} \norm{\rho_g}_{L^\infty}^2
 + \frac{C \Eps_a}{\delta}
 + C \, \sqrt{T} 
 \norm{f \vint{v}^{m+\beta}}_{L^\infty_{t,x,v}}^2} \norm{w}_{L^2_{t,x,v}}^2 \nn
\\
& \quad \,
 +\frac{C}{\tau_0 \sqrt{\delta}}
  \vpran{\Eps_a + \frac{T}{a}} \norm{\vint{v}^{m+\beta} f}_{L^\infty_{t,x,v}}^2
\norm{w}_{L^2_{t,x,v}}^2,
\end{align}
where we have denoted
\begin{align*}
 c_1 := 4^{-\min(\beta, 0)} \frac{c_0}{8}. 
\end{align*}
Choose the parameters in the following order: first we choose $\delta, \tau_0$ small such that 
\begin{align*}
 C\delta + \tau_0 < \frac{c_1}{8}.
\end{align*}
Then we choose $a$ small such that 
\begin{align*}
  \Eps_a < \frac{c_1}{100}, 
\qquad
  \frac{C\Eps_a}{\delta}
< \frac{1}{100},
\qquad
  \frac{C\Eps_a}{\tau_0}
< \frac{1}{100}.
\end{align*}
Finally we choose $T$ small such that
\begin{align*}
 \frac{CT}{a^5} + \sqrt{T} < \frac{c_1}{8},
\qquad
 C \sqrt{T} \norm{\rho_g}_{L^\infty}^2
 + C \, \sqrt{T} 
 \norm{f \vint{v}^{4+m+\beta}}_{L^\infty_{t,x,v}}^2
< \frac{1}{100},
\end{align*}
and
\begin{align*}
 \frac{CT}{\tau_0 a \sqrt{\delta}}
 \norm{\vint{v}^{m+\beta} f}_{L^\infty_{t,x,v}}^2
< \frac{1}{100}.
\end{align*}
These choices reduce~\eqref{est:full-toy} to
\begin{align*}
 \frac{1}{6} \norm{w}_{L^2_{t,x,v}}^2
&\leq
 -\frac{c_1}{8} 
 \sum_{n=0}^\infty
 \norm{2^{\frac{\beta n}{2}} \nabla_v \CalM_n w_n}_{L^2_{T,t,x,v}}^2
 + \frac{1}{10} \norm{w}_{L^2_{t,x,v}}^2,
\end{align*}
from which we obtain the desired result of $w = 0$.
\end{proof}


\section{Landau Equation with Viscosity}

We consider a viscous approximation of the inhomogeneous Landau equation with Coulomb potential 
\begin{align}\label{FPL}
\partial_t f + v \cdot \nabla_x f = \nu \Delta_v f  + Q_L ( f,f),  
\end{align}
where $Q_L$ is the Landau operator, defined as 
$$
Q_L ( f,f) :=   \textrm{div}_v \left( A[f] \nabla f - f \nabla a[f]\right),
$$
with 
\begin{align*}
A[f] (x,v,t)& := \frac{1}{8\pi} \int_{\mathbb{R}^3} \frac{\mathbb{P}(v-z) }{|v-z|} f(x,z,t)\;dz,\\
a[f] (x,v,t) &:= \frac{1}{4\pi} \int_{\mathbb{R}^3} \frac{1 }{|v-z|} f(x,z,t)\;dz,
\end{align*}
with the relation that $\textrm{div}_v A[f] = \nabla_v a[f]$.  A smooth solution to (\ref{FPL}) is defined as follows: 

\begin{definition} \label{D_existence}
A positive function $f$ is a classical solution to (\ref{FPL}) in $(0,T) \times \T^3 \times \R^{3}$ with initial data $f_{in} \in L^{\infty}_x L^\infty_{v,k}$ and $k>\frac{15}{2}$ if $f$ satisfies (\ref{FPL}) classically, $ \|\vint{v}^k f\|_{L^\infty_{t,x,v}}$ depends only on the initial data, and $ f \in C_{kin}^{2,\alpha}(\Omega)$ for any compact subset $\Omega \subset (0, T) \times \T^3 \times \mathbb{R}^{3}$.  
\end{definition} 
The space $C_{kin}^{2,\alpha}$ is an abbreviation for the kinetic representation of the classical H\"older spaces used for Schauder estimates $C_t^{1+\frac{\alpha}{2}}C_x^{\frac{2+\alpha}{3}}C_v^{2+\alpha}$. 

We prove the following two theorems.

\begin{theorem} \label{Thm_Landau_0}
Let $f(x,v,t)$ be a solution of (\ref{FPL}) as in Definition \ref{D_existence} in the time interval $[0,T_\ast]$. Let $\nu >0$ be any arbitrary positive constant. Then $f$ is unique in the class of continuous functions $f$ such that  $ \langle v \rangle ^m f \in  L^\infty_{t,x}(L^4_v \cap L^1_v)$, for  $m > 3$. 
\end{theorem}

\begin{theorem} \label{Thm_Landau}
Let $f(x,v,t)$ be a solution of (\ref{FPL}) as in Definition \ref{D_existence} in the time interval $[0,T_\ast]$. For any viscosity constant $\nu >0$ there exists a positive constant $c_0$ such that if the solution $f$ satisfies 
\begin{align}\label{init_cond_small}
\sup_{t\in[0,T_\ast]} \| \langle v \rangle ^m f\|_{ L^\infty_{x}(L^4_v \cap L^1_v)} \le c_0, \quad \textrm{  $m > 3$},
\end{align}
then  $f$ is unique in the class of functions $ \langle v \rangle ^m f \in  L^\infty_{t,x}(L^4_v \cap L^1_v)$. 
\end{theorem}

\begin{remark} 
Let $f(x,v,t)$ be a solution of (\ref{FPL}) with $f_{in}$ as initial data. If $f_{in}$ is such that 
 $
\| \langle v \rangle ^k f_{in}\|_{ L^\infty_{x,v}} \le \mathcal{C}
$
for some $k\ge 0$, there exists a small time interval $[0,T_*]$ during which the $ L^\infty_{x,v}$-norm of $\langle v \rangle ^k f$ has grown at most twice. This can be shown by a maximum principle argument on (\ref{FPL}). In this time interval condition \eqref{init_cond_small} holds by taking $k$ big enough depending on $m$ and $\mathcal{C}$ small enough depending on $c_0$ and $m$.




\end{remark}

The proofs of theorems \ref{Thm_Landau_0} and \ref{Thm_Landau} only differ in the treatment of the Landau diffusion terms (denoted in the proof by $I_1$ and $I_2$). If we assume continuity for the solutions in all variables $x,v,t$,  the regular part of the Landau diffusion together with any amount of artificial viscosity controls the commutator of $\mathcal{M}$ with the transport. If we work with merely bounded and integrable solutions, Theorem \ref{Thm_Landau}, we need an amount of artificial viscosity that depends on the initial data.

We start by rewriting (\ref{FPL}) for the function 
$$g := \langle v \rangle ^m f.$$
Working with $g$, instead of $f$, will be useful to estimate $A[f]$, $a[f]$ and their gradient only in terms of $L^2$ norms; we recall that  
\begin{align}\label{A_sup}
\| A[f]\|_{L^\infty_v(\mathbb{R}^3)} & \le \| a[f]\|_{L^\infty_v(\mathbb{R}^3)}  \le c(m)  \|   \langle v \rangle^{m}f\|_{L^2_v(\mathbb{R}^3)} \quad \textrm{for some $m>3/2$, } \\
\| \nabla a [f]\|_{L^\infty_v(\mathbb{R}^3)}&  \le c(m) \|   \langle v \rangle^{m}f\|_{L^4_v(\mathbb{R}^3)} \quad \textrm{for some $m>3/4$, } \label{nablaA_sup}\\
\| \nabla a [f]\|_{L^6_v(\mathbb{R}^3)} & \le   \| f\|_{L^2_v(\mathbb{R}^3)}.\label{nablaa_L6} 
\end{align}
Denote $\vint{\cdot, \cdot}$ as the inner product in $\R^3$. Since 
\begin{align*} 
 \langle v \rangle ^m  Q_L ( f,f) = &\; \textrm{div}_v \left( A[f] \nabla g - g \nabla a[f]\right) -m\; \textrm{div}_v \left( A[f]g \frac{v}{ \langle v \rangle ^2}\right)\\
 &  - m \left  \langle A[f]   \nabla g ,  \frac{v}{ \langle v \rangle ^2}\right  \rangle +m^2 g \left \langle A[f]  \frac{v}{ \langle v \rangle ^2}  ,  \frac{v}{ \langle v \rangle ^2}\right  \rangle + m \left \langle g \nabla a[f],  \frac{v}{ \langle v \rangle ^2}\right  \rangle,\\
 \langle v \rangle ^m \Delta_v f = &\; \Delta_v g - 2m \frac{v \cdot \nabla g }{\langle v \rangle ^2} + (m^2+2m)\frac{|v|^2g }{\langle v \rangle ^4} -3m \frac{ g }{\langle v \rangle ^2},
\end{align*}
the function $g$ solves
\begin{align}\label{FPLg}
\partial_t g + v \cdot \nabla_x g = \;&\textrm{div}_v \left( A[f] \nabla g - g \nabla a[f]\right) +\nu \Delta_v g \nn
\\
&-m\; \textrm{div}_v \left( A[f]g \frac{v}{ \langle v \rangle ^2}\right) - m \left  \langle A[f]   \nabla g ,  \frac{v}{ \langle v \rangle ^2}\right  \rangle \nonumber \\
& +m^2 g \left \langle A[f]  \frac{v}{ \langle v \rangle ^2}  ,  \frac{v}{ \langle v \rangle ^2}\right  \rangle + m \left \langle g \nabla a[f],  \frac{v}{ \langle v \rangle ^2}\right  \rangle \nonumber  
\\
&- 2m \nu  \;\frac{v \cdot \nabla g }{\langle v \rangle ^2} +\nu  (m^2+2m)\; \frac{|v|^2g }{\langle v \rangle ^4} -3\nu m \; \frac{ g }{\langle v \rangle ^2}.
\end{align}
\begin{proof} (Proof of Theorem \ref{Thm_Landau_0}.)
Let $f_1$ and $f_2$ be two smooth solutions to (\ref{FPL}) with the same initial data $f_{in}$.  The function 
$$
w := g_1-g_2, \quad g_i = \langle v \rangle ^m f_i,
\qquad i = 1, 2,
$$ 
solves
\begin{align*}
\partial_t w + v \cdot \nabla_x w = \;&\textrm{div}_v \left( A[f_1] \nabla w - w \nabla a[f_1]\right)+\textrm{div}_v \left( A[f_1-f_2] \nabla g_2 - g_2 \nabla a[f_1-f_2]\right)  +\nu \Delta_v w  \nonumber \\
&-m\; \textrm{div}_v \left( A[f_1]w \frac{v}{ \langle v \rangle ^2}\right)-m\; \textrm{div}_v \left( A[f_1-f_2]\; g_2 \frac{v}{ \langle v \rangle ^2}\right) \nonumber \\
&- m \left  \langle A[f_1]   \nabla w ,  \frac{v}{ \langle v \rangle ^2}\right  \rangle - m \left  \langle A[f_1-f_2]   \nabla g_2 ,  \frac{v}{ \langle v \rangle ^2}\right  \rangle \nonumber \\
& +m^2 w \left \langle A[f_1]  \frac{v}{ \langle v \rangle ^2}  ,  \frac{v}{ \langle v \rangle ^2}\right  \rangle +m^2 g_2 \left \langle A[f_1-f_2]  \frac{v}{ \langle v \rangle ^2}  ,  \frac{v}{ \langle v \rangle ^2}\right  \rangle  \nonumber \\
& + m w \left \langle  \nabla a[f_1],  \frac{v}{ \langle v \rangle ^2}\right  \rangle  + m \left \langle g_2 \nabla a[f_1-f_2],  \frac{v}{ \langle v \rangle ^2}\right  \rangle \nonumber  \\
&- 2m \nu  \;\frac{v \cdot \nabla w }{\langle v \rangle ^2} +\nu  (m^2+2m)\; \frac{|v|^2w }{\langle v \rangle ^4} -3\nu m \; \frac{ w }{\langle v \rangle ^2} \nonumber\\
=:&\; I_1 + ...+ I_{14}.
\end{align*}

Apply the $\CalM$-operator to the above equation and obtain 
\begin{align*}
  \del_t \CalM w  + v \cdot \nabla_x \CalM w
= [\CalM, \partial_t + v \cdot \nabla_x]w + \sum_{i=1}^{14} \CalM  I_i   ,
\end{align*}
with the transport commutator defined in~\eqref{trasp_comm}
We multiply the resulting equation by ${{\CalM} w}$ and integrate in $(0, T_0) \times (0,T) \times \T^3 \times \R^3$. This gives
\begin{align} \label{first_est}
\frac{1}{2}\int_0^{T_0} \int_0^{T}   \partial _t \iint_{\T^3 \times \R^3} |{\CalM} w|^2 
= & \int_0^{T_0}  \int_0^{T}  {\CalM} w  \iint_{\T^3 \times \R^3} [\CalM, \partial_t + v \cdot \nabla_x]w 
\nonumber 
\\
&\;+ \int_0^{T_0}  \int_0^{T}  \iint_{\T^3 \times \R^3}  {{\CalM}w} {\CalM} ( {I_1} + ...+  {I_{14}}). 
\end{align}

Since ${M}(\cdot,\cdot,T)=1$ and $\widehat w(\cdot,\cdot,0)=0$, the left hand side of (\ref{first_est}) reduces to 
$$
\frac{1}{2} \int_0^{T_0} \|\widehat w(T) \|_{L^2}^2 \dT = \frac{1}{2} \int_0^{T_0} \iint_{\T^3 \times \R^3}  w^2 \dx\dv\dt .
$$
The first term on the right hand side of (\ref{first_est}) can be estimated as 
\begin{align*}
 \delta \left(\frac{1}{2}+ \Eps\right)& \sum_{\eta}\int_0^{T_0} \int_0^T \int_{\R^3} \frac{ \langle \xi \rangle^{2}}{{ 1 + \delta \int_t ^T \langle \xi +(t- \tau) \eta\rangle^{2} \dtau}} |{M}\widehat w|^2 
\\
 \le  & \;\delta  \left(\frac{1}{2}+ \Eps\right)\sum_{\eta} \left( \int_0^{T_0} \int_0^T \int_{\R^3}   
 | \xi |^{2}  |{M}\widehat w|^2 
 + T_0 \int_0^{T_0} \int_{\R^3}  |\widehat w|^2 
 \right)
\\
 = & \;\delta  \left(\frac{1}{2}+ \Eps\right) \int_0^{T_0} \int_0^T \iint_{\T^3 \times \R^3} |\nabla_v \;\CalM w|^2 
 + T_0 \int_0^{T_0} \iint_{\T^3 \times \R^3}   w^2. 
\end{align*}
Moreover, 
$$
 \int_0^{T_0} \int_0^T   \iint_{\T^3 \times \R^3} \CalM w  \; {\CalM} { I_3} 
 = - \nu  \int_0^{T_0} \int_0^T  \iint_{\T^3 \times \R^3}   |\nabla_v \; \mathcal{M} w| ^2. 
$$

Since ${M}$ is uniformly bounded from above by $1$, we bound the term ${ I_{14}}$ as follows 
\begin{align*} 
& \quad \,
-3m \nu \int_0^{T_0}\int_0^{T} \iint_{\T^3 \times \R^3} { ( {\CalM} w)  }  {\CalM} { \frac{ w }{\langle v \rangle ^2} } 
\dx\dv\dt\dT
\\
&=3m \nu  \sum_{\eta} \int_0^{T_0}\int_0^{T}\int_{\R^3}  |M|^2 \mathcal{F}\left({ \frac{ w }{\langle v \rangle ^2} }\right) \mathcal{F}\left(w\right) \dxi \dt\dT 
\\
&\le 
3m \nu\int_0^{T_0}\int_0^{T}  \iint_{\T^3 \times \R^3} w^2 \dx\dv\dt\dT 
\leq 
3m \nu T_0 \int_0^{T_0}\iint_{\T^3 \times \R^3} w^2 \dx\dv\dt. 
\end{align*}
We carry out similar estimates for ${I_{13}}$.  To bound ${I_{12}}$ we first note that  
\begin{align*}
\frac{v \cdot \nabla w }{\langle v \rangle ^2} = \textrm{div} \left(\frac{v  w }{\langle v \rangle ^2} \right) - \frac{3w}{\langle v \rangle ^2} +  \frac{w|v|^2 }{\langle v \rangle ^2} =: J_1+J_2+J_3.
\end{align*}
The terms $J_2$ and $J_3$ are similar to $I_{14}$. For $J_1$ we have that 
\begin{align*}
& \quad \,
-2m\nu \int_0^{T_0}\int_0^{T} \iint_{\T^3 \times \R^3}  \CalM w \CalM \left( \textrm{div} \left(\frac{v  w }{\langle v \rangle ^2} \right)\right) \dx\dv\dt\dT 
\\
& = 2m\nu \int_0^{T_0}\int_0^{T} \iint_{\T^3 \times \R^3} \nabla \CalM w \cdot \CalM \left(\frac{v  w }{\langle v \rangle ^2} \right)  \dx\dv\dt\dT 
\\
& \leq \nu^2  \int_0^{T_0}\int_0^{T} \iint_{\T^3 \times \R^3} |\nabla \CalM w|^2 \dx\dv\dt\dT 
+ 4m^2 \int_0^{T_0}\int_0^{T} \iint_{\T^3 \times \R^3}  \left|\frac{v  w }{\langle v \rangle ^2} \right|^2  
\\
&\leq 
\nu^2  \int_0^{T_0}\int_0^{T} \iint_{\T^3 \times \R^3} |\nabla \CalM w|^2 \dx\dv\dt\dT 
+  4m^2 T_0 \int_0^{T_0} \iint_{\T^3 \times \R^3}  | w|^2  \dx\dv\dt.
\end{align*}
We now analyze $I_1=\textrm{div}_v \left( A[f_1] \nabla w - w \nabla a[f_1]\right)$.  Following the notation of Lemma \ref{lem:kernels}, we~write 
$$
\textrm{div}_v \left( A[f_1] \nabla w\right) = \textrm{div}_v \left( \vpran{A[f_1]\ast \psi_a} \nabla w\right) + \textrm{div}_v \left( (A[f_1]-A[f_1]\ast \psi_a )\nabla w\right).
$$
The matrix $ A[f_1]\ast \psi_a $ is nonnegative definite. In the term with the smooth diffusion we first integrate by parts and then commute $\mathcal{M}$ with $A[f_1]\ast \psi_a $. This yields  
\begin{align*}
& \quad \,
\int_0^{T_0}\int_0^{T} \iint_{\T^3 \times \R^3} \mathcal{M} w \; \mathcal{M} \; \textrm{div}_v \;\left( A[f_1]\ast \psi_a \nabla w \right)\dx\dv\dt\dT 
\\
&= -\int_0^{T_0}\int_0^{T} \iint_{\T^3 \times \R^3}  \nabla(\mathcal{M} w) \cdot  \; \mathcal{M} \; \;\left( A[f_1]\ast \psi_a \nabla w \right) \dx\dv\dt\dT 
\\
&= -\int_0^{T_0}\int_0^{T} \iint_{\T^3 \times \R^3}  \nabla(\mathcal{M} w) \cdot  \; [\mathcal{M}, A[f_1]\ast \psi_a] \nabla w \dx\dv\dt\dT
\\
& \quad \,
 -\int_0^{T_0}\int_0^{T} \iint_{\T^3 \times \R^3}  \nabla(\mathcal{M} w) \cdot  \; \left( A[f_1]\ast \psi_a \nabla (\mathcal{M}w) \right) \dx\dv\dt\dT.
\end{align*}
The last integral is nonnegative and can be discarded. For the first one, we use  Lemma \ref{lem:kernels} with $h:= A[f_1]\ast \psi_a$, Lemma \ref{lem:M-comm-x-v}, and Young's inequality to get 
\begin{align*}
& \quad \,
-\int_0^{T_0}\int_0^{T} \iint_{\T^3 \times \R^3}  \nabla(\mathcal{M} w) \cdot  \; [\mathcal{M}, A[f_1]\ast \psi_a] \nabla w \dx\dv\dt\dT 
\\
& \leq 
\frac{\nu}{20}\int_0^{T_0}\int_0^{T} \iint_{\T^3 \times \R^3}  |\nabla(\mathcal{M} w)|^2 
+ \frac{20}{\nu}\int_0^{T_0}\int_0^{T} \iint_{\T^3 \times \R^3} |\mathcal{T} \nabla  (\mathcal{M}w )|^2 
\\
& \leq 
\frac{\nu}{20}\int_0^{T_0}\int_0^{T} \iint_{\T^3 \times \R^3}  
|\nabla(\mathcal{M} w)|^2 
+ \frac{20 c(a) }{\nu}\int_0^{T_0}\int_0^{T} \iint_{\T^3 \times \R^3} w^2 
\\
&\le \left( \frac{\nu}{20} +  \frac{10 c(a)T_0 }{\nu}\right)  \int_0^{T_0}\int_0^{T} \iint_{\T^3 \times \R^3}  
|\nabla(\mathcal{M} w)|^2 
+ \frac{20 c(a) T_0}{\nu}\int_0^{T_0}\iint_{\T^3 \times \R^3} 
w^2.
\end{align*}
We now look at the non-smooth part of the diffusion: 
\begin{align*}
& \quad \,
\int_0^{T_0}\int_0^{T} \iint_{\T^3 \times \R^3} 
\mathcal{M} w \; \mathcal{M}\textrm{div}_v \;\left( (A[f_1]-A[f_1]\ast \psi_a) \nabla w \right)\dx\dv\dt\dT
\\
&= - \int_0^{T_0}\int_0^{T} \iint_{\T^3 \times \R^3}  \nabla (\mathcal{M} w) \cdot \mathcal{M} \;\left( (A[f_1]-A[f_1]\ast \psi_a) \nabla w \right)\dx\dv\dt\dT
\\
&= - \int_0^{T_0}\int_0^{T} \iint_{\T^3 \times \R^3}  \nabla (\mathcal{M} w) \cdot \textrm{div}_v \mathcal{M} \;\left( (A[f_1]-A[f_1]\ast \psi_a)  w \right) \dx\dv\dt\dT
\\
& \quad \,
- \int_0^{T_0}\int_0^{T} \iint_{\T^3 \times \R^3}  \nabla (\mathcal{M} w) \cdot \mathcal{M} \;\left( (\nabla a[f_1]-\nabla a[f_1]\ast \psi_a)  w \right)\dx\dv\dt\dT =: B_1+B_2.
\end{align*}
In $B_1$, Young's inequality and (\ref{lapl}) yield
\begin{align*}
    |B_1| 
& \leq   
 \frac{\nu}{20} \int_0^{T_0}\int_0^{T} \iint_{\T^3 \times \R^3}  |\nabla(\mathcal{M} w)|^2 \dx\dv\dt\dT 
\\
& \quad \, 
    + \frac{20}{\nu} \sum_\eta \int_0^{T_0}\int_0^{T} \int_{\R^3}  |\xi|^2 M^2 |  \mathcal{F}({(A[f_1]- A[f_1]\ast \psi_a) w})|^2  \dxi \deta \dt\dT 
\\
& \leq  
   \frac{\nu}{20} \int_0^{T_0}\int_0^{T} \iint_{\T^3 \times \R^3}  |\nabla(\mathcal{M} w)|^2 \dx\dv\dt\dT 
\\
    & \quad \,
    +  \frac{40}{\nu \Eps \delta} \int_0^{T_0} \iint_{\T^3 \times \R^3}  |  (A[f_1]- A[f_1]\ast \psi_a) w|^2  \dx \dv \dt
\\
& \leq  
\frac{\nu}{20} \int_0^{T_0}\int_0^{T} \iint_{\T^3 \times \R^3}  |\nabla(\mathcal{M} w)|^2 \dx\dv\dt\dT 
\\
& \quad \, 
+  \frac{40}{\nu \Eps \delta}\|  (A[f_1]- A[f_1]\ast \psi_a)\|_{L^{\infty}_{t,x,v}}^2 \int_0^{T_0} \iint_{\T^3 \times \R^3}  w^2  \dx \dv \dt ,
\end{align*}
with the abuse of notation that $|  \mathcal{F}({(A[\cdot]- A[\cdot]\ast \psi_a) w})|^2$ stands for $\sum_{i=1}^3 |\mathcal{F}(A_i)|^2$ with $A_i$ the $i$-th column of the matrix $(A[\cdot]- A[\cdot]\ast \psi_a)w$.  Since $A[f_1]\ast \psi_a-A[f_1]$ converges to zero uniformly  in $(x,v,t)$  as $a \to 0$ for $f_1$ continuous and decaying fast enough for large velocity, we can chose  $a$ small enough such that 
$$
 \frac{40}{\nu \Eps \delta} \|  (A[f_1]- A[f_1]\ast \psi_a)\|_{L^{\infty}_{t,x,v}}^2 \le  \frac{1}{8}. 
$$
 Summarizing, 
 \begin{align*}
    |B_1|  \; \le \; &    \frac{\nu}{20} \int_0^{T_0}\int_0^{T} \iint_{\T^3 \times \R^3}  |\nabla(\mathcal{M} w)^2 \dx\dv\dt\dT  + \frac{1}{8} \int_0^{T_0} \iint_{\T^3 \times \R^3}  w^2  \dx \dv \dt.
 \end{align*}
In $B_2$ we apply Young's inequality and ${M}\le 1$ to get 
\begin{align*}
    |B_2| 
& \leq  
\frac{\nu}{20}\; \int_0^{T_0}\int_0^{T} \iint_{\T^3 \times \R^3}  |\nabla(\mathcal{M} w)|^2 
+  \frac{20}{\nu} \int_0^{T_0}\int_0^{T} \iint_{\T^3 \times \R^3} 
\abs{\nabla a[f_1]-\nabla a[f_1]\ast \psi_a}^2  w^2   
\\
&\leq  
\frac{\nu}{20} \int_0^{T_0}\int_0^{T} \iint_{\T^3 \times \R^3}  |\nabla(\mathcal{M} w)|^2 
+  \frac{20 \;T_0 }{\nu} \| \nabla a[f_1]-\nabla a[f_1]\ast \psi_a\|_{L^\infty_{t,x,v}}^2 \int_0^{T_0} \iint_{\T^3 \times \R^3}   w^2   
\\
& \leq  
\frac{\nu}{20} \int_0^{T_0}\int_0^{T} \iint_{\T^3 \times \R^3}  
|\nabla(\mathcal{M} w)|^2 
+  \frac{20 c(m) T_0 }{\nu} \| g_1\|_{L^\infty_{t,x}(L^4_v)}^2 \int_0^{T_0} \iint_{\T^3 \times \R^3}   w^2.   
\end{align*}
The remaining part of $I_1$, namely $-\textrm{div}_v \left( w \nabla a[f_1]\right)$, can be estimated as $B_2$: 
\begin{align*}
& \quad \,
\int_0^{T_0}\int_0^{T} \iint_{\T^3 \times \R^3} \mathcal{M} w \; \mathcal{M}\textrm{div}_v \left( w \nabla a[f_1]\right) \dx\dv\dt\dT 
\\
& \leq  
\frac{\nu}{20}\; \int_0^{T_0}\int_0^{T} \iint_{\T^3 \times \R^3}  |\nabla(\mathcal{M} w)|^2 
+  \frac{20}{\nu} \int_0^{T_0}\int_0^{T} \iint_{\T^3 \times \R^3}   
|\nabla a[f_1]|^2  w^2   
\\
& \leq  
\frac{\nu}{20} \int_0^{T_0}\int_0^{T} \iint_{\T^3 \times \R^3}  
|\nabla(\mathcal{M} w)|^2 
+  \frac{20 c(m) T_0 }{\nu} \| g_1\|_{L^\infty_{t,x}(L^4_v)}^2 \int_0^{T_0} \iint_{\T^3 \times \R^3}   w^2.  
\end{align*}
All the above yields 
\begin{align*}
    \int_0^{T_0}\int_0^{T} \iint_{\T^3 \times \R^3} \mathcal{M} w \; \mathcal{M}I_1 
& \leq  
\frac{3 \nu}{20}\; \int_0^{T_0}\int_0^{T} \iint_{\T^3 \times \R^3}  
|\nabla(\mathcal{M} w)|^2 \dx\dv\dt\dT 
\\
& \quad \,
 + \left( \frac{20 c(m) T_0 }{\nu} \| g_1\|_{L^\infty_{t,x}(L^4_v)}^2 + \frac{1}{8}\right) \int_0^{T_0} \iint_{\T^3 \times \R^3}   w^2   \dx\dv\dt. 
\end{align*}
We now look at $I_2 = \textrm{div}_v \left(A[f_1-f_2]\nabla g_2 \right)- \textrm{div}_v (g_2 \nabla a[f_1-f_2])$; for the second part we have 
\begin{align}
& \quad \, 
\int_0^{T_0} \int_0^{T} \iint_{\T^3 \times \R^3}  \mathcal{M} w \;\textrm{div}_v ( \mathcal{M} g_2 \nabla a[f_1-f_2])q \dx\dv\dt\dT \nn
\\
&\leq \frac{\nu}{20}\int_0^{T_0} \int_0^{T} \iint_{\T^3 \times \R^3}  |\nabla(\mathcal{M} w)|^2 
+ \frac{20}{\nu} \sum_\eta \int_0^{T_0} \int_0^{T} \int_{\R^3}  {M}^2 \abs{\mathcal{F} (g_2 \nabla a[f_1-f_2]}^2 \nn 
\\
&\leq 
\frac{\nu}{20}\int_0^{T_0} \int_0^{T} \iint_{\T^3 \times \R^3}  |\nabla(\mathcal{M} w)|^2 
+  \frac{20}{\nu}\int_0^{T_0} \int_0^{T} \iint_{\T^3 \times \R^3}   g_2^2 |\nabla a[f_1-f_2]|^2 
\nonumber 
\\
&\leq \frac{\nu}{20}\int_0^{T_0} \int_0^{T} \iint_{\T^3 \times \R^3}  |\nabla(\mathcal{M} w)|^2 
+ \frac{20}{\nu}\|g_2\|^2_{L_{x,t}^\infty (L^3_v)}
 \int_0^{T_0} \int_0^{T}\int_{\T^3} \left(  \int_{\mathbb{R}^3 }|\nabla a[f_1-f_2] |^6 \;dv \right)^{1/3} 
\nonumber 
\\
&\leq 
\frac{\nu}{20}\int_0^{T_0} \int_0^{T} \iint_{\T^3 \times \R^3}  |\nabla(\mathcal{M} w)|^2 
+ \frac{20 c(m)}{\nu}\|g_2\|^2_{L_{x,t}^\infty (L^3_v)} 
 \int_0^{T_0} \int_0^{T}\iint_{\T^3 \times \R^3} |f_1-f_2 |^2 
 \nonumber 
\\
&\leq 
\frac{\nu}{20} \int_0^{T_0}\int_0^{T} \iint_{\T^3 \times \R^3}  
|\nabla(\mathcal{M} w)|^2 
+ \frac{20 c(m)  \|g_2\|^2_{L_{x,t}^\infty (L^3_v)}}{\nu}T_0  \int_0^{T_0}\iint_{\T^3 \times \R^3}  | w|^2, \label{I_2_partII} 
\end{align}
using (\ref{nablaa_L6}) and ${M}\le 1$. For the diffusion term, we apply the same decomposition between smooth and less smooth parts as in $I_1$, this time for $\nabla g_2$: 
$$
\textrm{div}_v \left( A[f_1-f_2] \nabla g_2\right) = \textrm{div}_v \left( A[f_1-f_2] \nabla g_2 \ast \psi_a \right) + \textrm{div}_v \left( (A[f_1-f_2] (\nabla g_2 - \nabla g_2 \ast \psi_a )\right).
$$
We estimate the smooth part as 
\begin{align*}
& \quad \,
   \int_0^{T_0} \int_0^{T} \iint_{\T^3 \times \R^3}  \mathcal{M} w \; \mathcal{M} \textrm{div}_v \left( A[f_1-f_2] \nabla g_2 \ast \psi_a \right) \dx\dv\dt\dT  
\\
& = - \int_0^{T_0} \int_0^{T} \iint_{\T^3 \times \R^3}  \nabla(\mathcal{M} w) \cdot  \mathcal{M} \left( A[f_1-f_2] g_2 \ast \nabla \psi_a \right) \dx\dv\dt\dT 
\\
& \leq  
\frac{\nu}{20}\int_0^{T_0} \int_0^{T} \iint_{\T^3 \times \R^3}  |\nabla(\mathcal{M} w)|^2 
+  \frac{20}{\nu} \sum_\eta \int_0^{T_0} \int_0^{T} \int_{\R^3} M^2 | \mathcal{F}( A[f_1-f_2] g_2 \ast \nabla \psi_a)|^2 
\\
&\leq 
\frac{\nu}{20}\int_0^{T_0} \int_0^{T} \iint_{\T^3 \times \R^3}  |\nabla(\mathcal{M} w)|^2 
+  \frac{20}{\nu} \int_0^{T_0} \int_0^{T} \iint_{\T^3 \times \R^3}  
|  A[f_1-f_2] g_2 \ast \nabla \psi_a|^2,
\end{align*}
using once more that $M\le 1$. In the last integral, we use the smoothness of $g_2 \ast \nabla \psi_a$ to obtain 
\begin{align*}
   \int_0^{T_0} \int_0^{T}& \iint_{\T^3 \times \R^3}  |  A[f_1-f_2] g_2 \ast \nabla \psi_a|^2 \dx\dv \dt\dT
\\
& \leq   \int_0^{T_0} \int_0^{T} \int_{\T^3} \left(\| A[f_1-f_2] \|^2_{L^\infty_v} \int_{\R^3}  |g_2 \ast \nabla \psi_a|^2 \dv\right) \dx\dt\dT
\\
& \leq  
c(a) c(m) \int_0^{T_0} \int_0^{T} \int_{\T^3} \| w \|^2_{L^2_v} \|g_2\|^2_{L^1_v} \dx\dt\dT
\\
& \le  
c(a) c(m) \|g_2\|^2_{L^\infty_{t,x}L^1_v}  \int_0^{T_0} \int_0^{T} \iint_{\T^3 \times \R^3} w^2 \dx\dv\dt\dT
\\
& \le  
c(a) c(m)  T_0 \|g_2\|^2_{L^\infty_{t,x}L^1_v}  \int_0^{T_0} \iint_{\T^3 \times \R^3} w^2 \dx\dv\dt. 
\end{align*}
Hence we get 
\begin{align*}
& \quad \,
    \int_0^{T_0} \int_0^{T} \iint_{\T^3 \times \R^3}   \mathcal{M} w \; \mathcal{M} \textrm{div}_v \left( A[f_1-f_2] \nabla g_2 \ast \psi_a \right) \; \dx\dv\dt\dT 
\\
&\leq  
\frac{\nu}{20}\int_0^{T_0} \int_0^{T} \iint_{\T^3 \times \R^3}  
|\nabla(\mathcal{M} w)|^2 
+  \frac{20}{\nu}  c(a) \, T_0 \, \|g_2\|^2_{L^\infty_{t,x}L^1_v}  \int_0^{T_0}  \iint_{\T^3 \times \R^3} w^2. 
\end{align*}
The less smooth part can be estimated as 
\begin{align*}
   \int_0^{T_0} \int_0^{T}& \iint_{\T^3 \times \R^3}    \mathcal{M} w \; \mathcal{M} \textrm{div}_v \left( A[f_1-f_2](\nabla g_2 - \nabla g_2 \ast \psi_a )\right) \; \dx\dv\dt\dT  
\\
   =& - \int_0^{T_0} \int_0^{T} \iint_{\T^3 \times \R^3}  \nabla(\mathcal{M} w) \cdot  \mathcal{M} \left( A[f_1-f_2] (\nabla g_2 - \nabla g_2 \ast \psi_a) \right) \dx\dv\dt\dT 
\\
&= - \int_0^{T_0} \int_0^{T} \iint_{\T^3 \times \R^3}  \nabla(\mathcal{M} w) \cdot  \textrm{div}_v ( \mathcal{M} \left( A[f_1-f_2] ( g_2 -g_2 \ast \psi_a) \right) \dx\dv\dt\dT 
\\
& +  \int_0^{T_0} \int_0^{T} \iint_{\T^3 \times \R^3}  \nabla(\mathcal{M} w) \cdot   \mathcal{M} \vpran{ \nabla a[f_1-f_2] \; ( g_2 -g_2 \ast \psi_a)}  \dx\dv\dt\dT. 
\end{align*}
The second term can be estimated as in (\ref{I_2_partII}). For the first, we apply Young's inequality and (\ref{int_xi_M2}) to get 
\begin{align*}
    - \int_0^{T_0}& \int_0^{T} \iint_{\T^3 \times \R^3}  \nabla(\mathcal{M} w) \cdot  \textrm{div}_v ( \mathcal{M} \left( A[f_1-f_2] ( g_2 -g_2 \ast \psi_a) \right) \dx\dv\dt\dT 
\\
    \le  \; & \frac{\nu}{20}\int_0^{T_0} \int_0^{T} \iint_{\T^3 \times \R^3}  |\nabla(\mathcal{M} w)|^2 \dx\dv\dt\dT  
\\
   &  + \frac{40}{\nu \Eps \delta } \int_0^{T_0} \iint_{\T^3 \times \R^3} | A[f_1-f_2] ( g_2 -g_2 \ast \psi_a)|^2 \dx\dv\dt 
\\
 \le  \; & \frac{\nu}{20}\int_0^{T_0} \int_0^{T} \iint_{\T^3 \times \R^3}  |\nabla(\mathcal{M} w)|^2 \dx\dv\dt\dT  
\\
   &  +  \frac{20 c(m)}{\nu \Eps \delta} \|  (g_2 -g_2 \ast \psi_a)\|^2_{L^\infty_{t,x}(L^2_v)} \int_0^{T_0} \iint_{\T^3 \times \R^3} w^2 \dx\dv\dt    
\\
    \le  \; & \frac{\nu}{20}\int_0^{T_0} \int_0^{T} \iint_{\T^3 \times \R^3}  |\nabla(\mathcal{M} w)|^2 
+  \frac{1}{8 } \int_0^{T_0} \iint_{\T^3 \times \R^3} w^2, 
\end{align*}
for $a$ small enough, thanks to the uniform convergence of $g_2 -g_2 \ast \psi_a$ to zero as $a \to 0$.
Summarizing, we have 
\begin{align*}
    \int_0^{T_0}\int_0^{T} \iint_{\T^3 \times \R^3} \mathcal{M} w \; \mathcal{M}I_2 \dx\dv\dt\dT 
\le & \;  
\frac{ \nu}{5}\; \int_0^{T_0}\int_0^{T} \iint_{\T^3 \times \R^3}  |\nabla(\mathcal{M} w)|^2 \dx\dv\dt\dT 
\\
    &+ \left(c \, T_0+ \frac{1}{8}\right) \int_0^{T_0} \iint_{\T^3 \times \R^3}   w^2  \dx\dv\dt,
\end{align*}
where $c$ depends on $\nu$, some $L^\infty_{t,x}(L^2_v)$ norms of $g_i$ and other universal constants. 


The estimates of ${I_4}$ to ${I_{11}}$ are similar, and we only sketch them briefly. For ${I_4}$ we have 
\begin{align*}
   \int_0^{T_0}\int_0^{T} \iint_{\T^3 \times \R^3}  \CalM w \CalM I_4 
= m  \int_0^{T_0}\int_0^{T} \iint_{\T^3 \times \R^3}  \nabla (\CalM w ) \CalM\left(\frac{A[f_1]v  w }{\langle v \rangle ^2} \right) 
\\
\le  
\frac{\nu}{10}  \int_0^{T_0}\int_0^{T} \iint_{\T^3 \times \R^3}  |\nabla (\CalM w )| ^2 
+ \frac{10 \, c(m) m^2 \, T_0 \, \|g_1\|^2_{L^\infty_{t,x}L^2_v}}{\nu}  \int_0^{T_0} \iint_{\T^3 \times \R^3}  | {w} |^2, 
\end{align*}
and similar estimates hold for the term with $I_5$:  
\begin{align*}
& \quad \,
\int_0^{T_0}\int_0^{T} \iint_{\T^3 \times \R^3} 
 {({\CalM}w)} {\CalM} I_5 \dx\dv\dt\dT
\\ 
& \le \;  \frac{\nu}{10}  \int_0^{T_0}\int_0^{T} \iint_{\T^3 \times \R^3}  
|\nabla (\CalM w )| ^2 
+ \; \frac{10 \, m^2}{\nu}  \int_0^{T_0}  \int_0^T \iint_{\T^3 \times \R^3} 
|A[f_1-f_2]  g_2|^2 
\\
&\le \;  \frac{\nu}{10}  \int_0^{T_0}\int_0^{T} \iint_{\T^3 \times \R^3}  |\nabla (\CalM w )| ^2 
+ \; \frac{10 \, c(m)m^2 \, T_0 \, \|g_2\|^2_{L^\infty_{t,x}L^2_v}}{\nu}  \int_0^{T_0} \iint_{\T^3 \times \R^3}  |w|^2. 
\end{align*}
For $I_{11}$ we also use the fact that $ {M} \le 1$ to get 
\begin{align*}
& \quad \,
\int_0^{T_0}\int_0^{T} \iint_{\T^3 \times \R^3} { ( {\CalM} w)  } \CalM {I_{11}} \dx\dv\dt\dT  
\\
&\le 
\int_0^{T_0}\int_0^{T} \iint_{\T^3 \times \R^3} |w| ^2 \dx\dv\dt\dT 
+ \int_0^{T_0}\int_0^{T} \iint_{\T^3 \times \R^3} \left| I_{11}\right|^2 \dx\dv\dt\dT
\\
&\le 
T_0 \int_0^{T_0} \iint_{\T^3 \times \R^3}  | w| ^2 \dx\dv\dt  
+  T_0 \int_0^{T_0} \iint_{\T^3 \times \R^3}  g_2^2 | \nabla a[f_1-f_2]|^2 \dx\dv\dt 
\\
&\le   
T_0 \int_0^{T_0} \iint_{\T^3 \times \R^3} |w| ^2 
+  T_0 \int_0^{T_0} \int_{\mathbb{R}^3} \|g_2\|_{L_v^3}^2\left( \int_{\mathbb{R}^3} | \nabla a[f_1-f_2]|^6 \dv \right)^{1/3} \dx\dt
\\
&\le   
T_0 \vpran{1+ c(m) \|g_2\|^2_{L^\infty_{x,t}L_v^3}} \int_0^{T} \iint_{\T^3 \times \R^3} w^2 \dx\dv\dt.
\end{align*}
Similarly, for the term with $I_9$ we have
\begin{align*}
\int_0^{T_0}\int_0^{T} & \iint_{\T^3 \times \R^3} {({\CalM} w)  } \CalM {I_{9}} \dx\dv\dt\dT    
\\
&\le 
T_0 \int_0^{T_0} \iint_{\T^3 \times \R^3} |w| ^2 \dx\dv\dt 
+  T_0 \int_0^{T_0} \iint_{\T^3 \times \R^3}  g_2^2 | A[f_1-f_2]|^2 \dx\dv\dt 
\\
&\le 
T_0 \int_0^{T_0} \iint_{\T^3 \times \R^3} |w| ^2 \dx\dv\dt 
+  T_0 c(m) \|g_2\|^2_{L^\infty_{x,t}L_v^2} \int_0^{T_0} \iint_{\T^3 \times \R^3} w^2 \dx\dv\dt
\\
&\le T_0 \vpran{1+c(m) \|g_2\|^2_{L^\infty_{t,x}L_v^2}} \int_0^{T_0} \iint_{\T^3 \times \R^3} w^2 \dx\dv\dt,
\end{align*}
using (\ref{A_sup}) to estimate $A[f_1-f_2]$. 
We rewrite $I_6$ and $I_7$ as 
\begin{align*}
-\left  \langle A[f_1]   \nabla w ,  \frac{v}{ \langle v \rangle ^2}\right \rangle = &\;   {{w\; \textrm{div}_v\left( A[f_1] \frac{v}{ \langle v \rangle ^2}   \right)}}
- \textrm{div} \left(  w A[f_1]    \frac{v}{ \langle v \rangle ^2} \right) 
\\
=&:\; w I_{6,1}- I_{6,2},
\\
 -  \left  \langle A[f_1-f_2]   \nabla g_2 ,  \frac{v}{ \langle v \rangle ^2}\right  \rangle = &\;   {{g_2\; \textrm{div}_v\left( A[f_1-f_2] \frac{v}{ \langle v \rangle ^2}   \right)}}
- \textrm{div} \left(  g_2 A[f_1-f_2]    \frac{v}{ \langle v \rangle ^2} \right) 
\\
=\; g_2 \bigg(\nabla a [f_1-f_2] & \frac{v}{ \langle v \rangle ^2} + Tr(A[f_1-f_2] \nabla \frac{v}{ \langle v \rangle ^2})\bigg) - \textrm{div} \left(  g_2 A[f_1-f_2]    \frac{v}{ \langle v \rangle ^2} \right) 
\\
=: \;  g_2 I_{7,1} + g_2 & \, I_{7,2} -  I_{7,3}.
\end{align*}
The term $I_{6,2}$ is similar to $I_{4}$.  For $I_{6,1}$ we use the bound $ {M} \le 1$ and get
\begin{align*}
\int_0^{T_0}\int_0^{T} & \iint_{\T^3 \times \R^3} \abs{\CalM w} |{\CalM}  \vpran{w I_{6,1}}| 
\leq 
\left(\int_0^{T_0}\int_0^{T} \iint_{\T^3 \times \R^3} | w| ^2 \right)^{1/2} 
\left( \int_0^{T_0}\int_0^{T} \iint_{\T^3 \times \R^3}  |{w I_{6,1}}|^2 
\right)^{1/2}   
\\
& \le  
\|I_{6,1}\|_{L^\infty_{t,x,v}} \int_0^{T_0}\int_0^{T} \iint_{\T^3 \times \R^3}  |w|^2 \dx\dv\dt\dT  
\\
&\le \|A[f_1] + \nabla a[f_1]\|_{L^\infty_{t,x,v}} T_0   \int_0^{T_0}\iint_{\T^3 \times \R^3}  |w| ^2 \dx\dv\dt
\\
&\le c(m) \|g_1\|_{L^\infty_{t,x}(L_v^1\cap L_v^4)} T_0 \int_0^{T_0}\iint_{\T^3 \times \R^3}  |w| ^2 \dx\dv\dt. 
\end{align*}
The term $g_2 I_{7,1}$ can be estimated as $I_{11}$, $g_2 I_{7,2}$ as $I_9$, and $I_{7,3}$ as $I_5$.  
Finally, $I_8$ and $I_{10}$ can be estimated as $w I_{6,1}$. Summarizing all the estimates for $I_1,...,I_{14}$ we get 
\begin{align*}
\left(\frac{1}{4}-cT_0\right) \int_0^{T_0} \iint_{\T^3 \times \R^3} w^2 \dx\dv\dt 
\le 
-\left(\frac{\nu}{10}-\delta\left(\frac{1}{2}+\Eps\right)\right) \int_0^{T_0}\int_0^{T} \iint_{\T^3 \times \R^3} |\nabla(\CalM w)|^2, 
\end{align*}
for $c$ a constant depending on $m$, $\nu$, $\Eps$, $\delta$ and $\|g_i\|_{L^\infty_{t,x}(L^1_v\cap L^4_v)}$. We pick $\delta$ small enough such that $\left(\frac{\nu}{10}-\delta\left(\frac{1}{2}+\Eps\right)\right)\le 0$ and then $T_0$ small enough so that $\frac{1}{4}-cT_0\ge 0$ to conclude that $w = 0$ for all times $t\in [0,T_0]$. 
\end{proof}

Finally, we show that if the initial data is small compared to the viscosity term, then the continuity of solutions is not needed and uniqueness holds for merely bounded solutions. 

\begin{proof}  (Proof of Theorem \ref{Thm_Landau}.)
We start with (\ref{first_est}). The terms $I_3, ... , I_{14}$ are estimated in the same way as in Theorem  \ref{Thm_Landau_0}. Because we work with bounded solutions, we need to use different methods for $I_1$ and $I_2$. 
Here, we rewrite $I_1$ as 
$$
\textrm{div}_v \left( A[f_1] \nabla w - w \nabla a[f_1]\right) 
= \nabla \otimes \nabla : ( A[f_1] w)  - 2 \textrm{div}_v (w \nabla a[f_1]) ,
$$
and use (\ref{int_xi_M2}), (\ref{A_sup}) and  (\ref{nablaA_sup}) to bound ${{\CalM} w} {\CalM} { I_1}$ using Young's inequality:
\begin{align*}
\int_0^{T_0} & \int_0^{T} \iint_{\T^3 \times \R^3} {{\CalM} w} {\CalM} {\nabla \otimes \nabla :( A[f_1] w)} \dx\dv\dt\dT 
\\
&\leq 
\frac{\nu}{10}\int_0^{T_0}\int_0^{T} \iint_{\T^3 \times \R^3} 
|\nabla(\CalM w)|^2 
+ \frac{10}{\nu} \sum_{\eta}\int_0^{T_0}\int_0^{T} \int_{\R^3}  |\xi {M}|^2 \abs{\widehat{A[f_1] w}}^2 
\\
&\leq   
\frac{\nu}{10}\int_0^{T_0}\int_0^{T} \iint_{\T^3 \times \R^3} 
|\nabla(\CalM w)|^2 
+ \frac{20}{ \delta \Eps \nu} \int_0^{T_0}\iint_{\T^3 \times \R^3}  |A[f_1] w|^2 
\\
&\leq   
\frac{\nu}{10}\int_0^{T_0}\int_0^{T} \iint_{\T^3 \times \R^3} 
|\nabla(\CalM w)|^2 
+ 20 c(m)  \frac{\|g_1\|^2_{L_{t,x}^\infty L^2_v}}{\delta \epsilon \nu} 
\int_0^{T_0} \iint_{\T^3 \times \R^3}|w|^2, 
\end{align*} 
and
\begin{align*}
- 2 \int_0^{T_0} & \int_0^{T} \iint_{\T^3 \times \R^3} 
\CalM w \CalM \textrm{div}_v({w \nabla a[f_1]}) \dx\dv\dt\dT 
\\
& \le  
\frac{\nu}{10}\int_0^{T_0}\int_0^{T} \iint_{\T^3 \times \R^3} |\nabla({M} w)| ^2 
+ \frac{10}{\nu} \int_0^{T_0}\int_0^{T} \iint_{\T^3 \times \R^3} 
w^2 |\nabla{a[f_1]}|^2 
\\
&\le   
\frac{\nu}{10}\int_0^{T_0}\int_0^{T} \iint_{\T^3 \times \R^3} 
|\nabla(\CalM w)|^2 
+ \frac{10 \, c(m) \, \|g_1\|^2_{L_{t,x}^\infty L^4_v}}{\nu} T_0 \int_0^{T_0}\iint_{\T^3 \times \R^3}  |w|^2 .
\end{align*}
Similarly, we rewrite $I_2$ as 
$$
I_2 = \nabla \otimes \nabla : ( A[f_1-f_2] g_2)  - 2 \textrm{div}_v (g_2 \nabla a[f_1-f_2]),
$$
and estimate
\begin{align*}
\int_0^{T_0} & \int_0^{T} \iint_{\T^3 \times \R^3} {{\CalM} w} {\CalM} {\nabla \otimes \nabla : ( A[f_1-f_2] g_2)} \dx\dv\dt\dT  
\\
&\leq  
\frac{\nu}{10}\int_0^{T_0}\int_0^{T} \iint_{\T^3 \times \R^3} 
|\nabla(\CalM w)| ^2 
+ \frac{10}{\nu} \sum_{\eta}\int_0^{T_0}\int_0^{T} \int_{\R^3}  |\xi|^2 {M}^2  |\CalF \vpran{A[f_1-f_2] g_2}|^2 
\\
&\le  
\frac{\nu}{10}\int_0^{T_0}\int_0^{T} \iint_{\T^3 \times \R^3} |\nabla(\CalM w)| ^2 
+ \frac{20}{ \delta \Eps \nu} \int_0^{T_0}\iint_{\T^3 \times \R^3}  |A[f_1-f_2] g_2|^2 
\\
&\le   
\frac{\nu}{10}\int_0^{T_0}\int_0^{T} \iint_{\T^3 \times \R^3} |\nabla(\CalM w)| ^2 
+ 20 c(m) \frac{\|g_2\|^2_{L_{t,x}^\infty L^2_v}}{ \delta \Eps \nu} \int_0^{T_0}\iint_{\T^3 \times \R^3} |w|^2, 
\end{align*} 
and 
\begin{align*}
& \quad \, 2 \int_0^{T_0} \int_0^{T} \iint_{\T^3 \times \R^3}  {( {\CalM} w)  } {\CalM}{\textrm{div}_v(g_2 \nabla a[f_1-f_2])}  \dx\dv\dt\dT  
\\
& \le  
\frac{\nu}{10}\int_0^{T_0}\int_0^{T} \iint_{\T^3 \times \R^3} |\nabla(\CalM w)| ^2 
+ \frac{10}{\nu} \int_0^{T_0}\int_0^{T} \iint_{\T^3 \times \R^3} g_2^2 \, |\nabla{a[f_1-f_2]}|^2 
\\
&\le  
\frac{10}{\nu} \int_0^{T_0}\int_0^{T} \int_{\T^3}  \left(\int_{\mathbb{R}^3 }|\nabla a[f_1-f_2] |^6 \dv \right)^{1/3}  \left(  \int_{\mathbb{R}^3 } g_2^3 \dv \right)^{2/3} 
\\
& \quad \,
+ \frac{\nu}{10}\int_0^{T_0}\int_0^{T} \iint_{\T^3 \times \R^3} |\nabla(\CalM w)| ^2 
\\
&\le  
\frac{\nu}{10}\int_0^{T_0}\int_0^{T} \iint |\nabla(\CalM w)|^2 
+ \frac{10 c(m) \|g_2\|^2_{L_{x,t}^\infty (L^3_v)}}{\nu} T_0 \int_0^{T_0} \iint_{\T^3 \times \R^3} | f_1-f_2|^2 
\\
&\le 
\frac{\nu}{10}\int_0^{T_0}\int_0^{T} \iint_{\T^3 \times \R^3} |\nabla(\CalM w)| ^2 
+ \frac{10 c(m)\|g_2\|^2_{L_{t,x}^\infty L^3_v}}{\nu}T_0  \int_0^{T_0} \iint_{\T^3 \times \R^3}  | w|^2, 
\end{align*}
using (\ref{int_xi_M2}), (\ref{A_sup}), (\ref{nablaA_sup}) and~\eqref{nablaa_L6}. 
Summarizing, we have that
\begin{align*}
& \quad \,
\frac{1}{2} \int_0^{T_0}\iint_{\T^3 \times \R^3}  {w}^2
\\
& \le 
\left[ -\frac{\nu}{10}+ \delta\left( \frac{1}{2} + \Eps \right)  \right] \int_0^{T_0}  \int_0^{T} \iint_{\T^3 \times \R^3} |\nabla(\CalM w)|^2 
\\
& \quad \, 
+ T_0\max \left\{ \frac{1}{\nu}, 1\right\} \left( 1 +c(m) \|g_1\|^2_{L^\infty_{t,x}(L_v^1\cap L_v^4)} + c(m) \|g_2\|^2_{L^\infty_{t,x}(L_v^1\cap L_v^4)} \right)
\int_0^{T_0} \iint_{\T^3 \times \R^3} {w}^2 
\\
& \quad \,
+ 2 c(m) \frac{\|g_1\|^2_{L^\infty_{t,x}(L_v^1\cap L_v^4)} + \|g_2\|^2_{L^\infty_{t,x}(L_v^1\cap L_v^4)}}{\delta \Eps \nu} \int_0^{T_0} \iint_{\T^3 \times \R^3} {w}^2. 
\end{align*}
For $\Eps \le 1$, let us choose $\delta$ and $T_0$ small enough such that 
$$
 -\frac{\nu}{10}+ \delta\left( \frac{1}{2} + \Eps \right)  < \frac{1}{2},
$$
and 
$$
 T_0\max \left\{ \frac{1}{\nu}, 1\right\} \left( 1 + c(m)\|g_1\|^2_{L^\infty_{t,x}(L_v^1\cap L_v^4)} + c(m) \|g_2\|^2_{L^\infty_{t,x}(L_v^1\cap L_v^4)}\right) < \frac{1}{4}. 
$$
Moreover, if the constant $c_0$ in \eqref{init_cond_small} is such that, for $T_0$ small,    
\begin{align}\label{cond_small}
2 c(m) \frac{\|g_1\|^2_{L^\infty_{t,x}(L_v^1\cap L_v^4)} + \|g_2\|^2_{L^\infty_{t,x}(L_v^1\cap L_v^4)}}{\delta \Eps \nu}   < \frac{1}{4},
\end{align}
then 
$$
\int_0^{T_0} \int w^2\; \dx\dv\dT =0, 
$$
which concludes the proof.     
\end{proof}

\subsection*{Acknowledgments} RA thanks the Qatar Foundation for the support of Grant 470252-25650.  MPG is partially supported by the DMS-NSF 2019335. WS is partially supported by NSERC Discovery Grant R832717. The authors thank the Isaac Newton Institute (INI) in Cambridge, UK, for their kind hospitality. They started this collaboration during the INI program {\it Frontiers in kinetic theory: connecting microscopic to macroscopic scales}, in 2022.

\bibliographystyle{amsxport}
\bibliography{kinetic}

\end{document}